\documentclass[11pt,reqno]{amsart}
\usepackage[english]{babel}

\usepackage{tikz}
\usetikzlibrary{positioning}
\usepackage{xcolor} 

\usepackage{amsmath, amssymb,amsthm}
\usepackage{amsfonts}
\usepackage{comment}
\usepackage[margin=1.0in]{geometry}
\usepackage{enumerate}
\usepackage{graphicx}
\usepackage{hyperref}
\usepackage{dsfont,bbm}
\usepackage{cite}
\usepackage{amstext}
\usepackage{mathrsfs}

\theoremstyle{plain}
\newtheorem{theorem}{Theorem}[section]
\newtheorem{definition}[theorem]{Definition}
\newtheorem{corollary}[theorem]{Corollary}
\newtheorem{proposition}[theorem]{Proposition}
\newtheorem{lemma}[theorem]{Lemma}
\newtheorem{remark}[theorem]{Remark}

\numberwithin{theorem}{section}
\numberwithin{equation}{section}

\newcommand{\N}{\mathbb{N}}

\newcommand{\R}{\mathbb{R}}

\newcommand{\cB}{{\mathcal B}}

\newcommand{\cE}{{\mathcal E}}
\newcommand{\cF}{{\mathcal F}}
\newcommand{\cG}{{\mathcal G}}
\newcommand{\cH}{{\mathcal H}}

\newcommand{\cJ}{{\mathcal J}}
\newcommand{\cK}{{\mathcal K}}

\newcommand{\cO}{{\mathcal O}}
\newcommand{\cP}{{\mathcal P}}

\newcommand{\cV}{{\mathcal V}}

\subjclass[2020]{
49Q10, 
35B10, 
35J93, 
34K18, 
46G05, 
82D60, 
}

\makeatother
\keywords{Screened Coulomb potential, Otha--Kawasaki energy, Diblock copolymer melt, Delaunay--type surfaces, Equilibrium periodic patterns.}

\begin{document}

\title[Delaunay Interfaces in screened diblock copolymers]
{Delaunay-type  interface in a screened  model of diblock copolymer melts}

\author{Guy Foghem$^{\dagger}$}
\address{{\small{G. Foghem: Brandenburgische Technische Universit\"at Cottbus--Senftenberg, Fakult\"{a}t 1: MINT Fachgebiet Mathematik, Platz der Deutschen Einheit 1, 03046 Cottbus, Germany.} \href{https://orcid.org/0000-0002-8917-7309}{ORCID}, Email:{guy.foghem[at]b-tu.de} }}

\thanks{$\dagger$ Financial support from the Deutsche Forschungsgemeinschaft (DFG) through the Walter Benjamin Programme (project FO~1699/1-1) is gratefully acknowledged.}

\author{Mouhamed Moustapha Fall$^{\ddagger}$}
\address{\small{M. M. Fall: African Institute for Mathematical Sciences in Senegal, KM 2, Route de Joal, B.P. 14 18. Mbour, Senegal.}  \href{https://orcid.org/0000-0003-4983-1895}{ORCID}, Email: {mouhamed.m.fall[at]aims-senegal.org}}

\thanks{$\ddagger$ This work was supported by the African Institute for Mathematical Sciences (AIMS), Senegal.}

\begin{abstract}
A diblock copolymer is a soft-matter composed of two chemically distinct block of repeating monomers  covalently bonded together at an end-to-end junction to form a single polymer chain. In this paper, we establish the existence of infinitely many smooth periodic unbounded domain patterns of Delaunay-type in $\mathbb{R}^3$ that optimize the energy distribution in diblock copolymer melts. We emphasize that pattern domains at the equilibrium correspond to stationary sets of the screened Ohta--Kawasaki free energy functional
\begin{align*}
\mathcal{P}_\gamma(\Omega) := |\partial\Omega| + \gamma \int_{\Omega}\int_{\Omega} G_{\kappa}(|x-y|) \,\mathrm{d}x\mathrm{d}y,
\end{align*}
where $\gamma>0$, $\kappa>0$ and  $G_\kappa(r)=\frac{1}{r} e^{-\kappa r}$ is the repulisive Yukawa potential. Equivalently, these equilibria satisfy the corresponding Euler--Lagrange equation
\begin{align*}
\cH_\Omega (x):= H_{\partial\Omega}(x) + \gamma \int_{\Omega} G_{\kappa}(|x-y|) \mathrm{d}y = \textrm{Const} \quad \text{on } \partial\Omega,
\end{align*}
where $H_{\partial\Omega}$ denotes the mean curvature of the surface $\partial\Omega$.
By analyzing the linearization of $\Omega \mapsto \mathcal{H}_\Omega$ around flat cylinders and applying the Crandall--Rabinowitz bifurcation theorem, for any $\kappa > 0$ and sufficiently small $\gamma > 0$, we prove the existence  of non-trivial, $2\pi$-periodic Delaunay-type equilibrium cylinder interfaces with shapes close to a Delaunay unduloid surface of constant mean curvature.
\end{abstract}

\maketitle
\vspace{-2mm}
\section{Introduction and main result}

A monomer is a molecule that can chemically bind with others to form a large macromolecular structure called a polymer. Polymers \cite{ASBXB20,RIES03,Rus96} are classified either as homopolymers, composed of a single type of repeating monomer unit inked together to form a monoblock 
or as copolymers, composed of two or more different monomer types. Copolymers formed from two monomers are termed bipolymers,  those formed from three are termed terpolymers, etc$\cdots$.
Beside this, a  diblock copolymer is a special type of polymer
consisting of exactly two different homopolymer blocks synthesized end-to-end point to form a single polymer chain.
\begin{figure}[h]
\centering

\begin{tikzpicture}[
A/.style={circle,draw=red!70!black,fill=red!80,minimum size=5mm,inner sep=0pt},
B/.style={circle,draw=blue!70!black,fill=blue!80,minimum size=5mm,inner sep=0pt},
bond/.style={red!70!black,line width=1.2pt},
bbond/.style={blue!70!black,line width=1.2pt},
font=\sffamily\bfseries
]

\begin{scope}[scale=0.73, shift={(0,0)}]


\node[left] at (-1.8,0) {\small Monoblock};

\foreach \i in {0,...,5}
{
    \node[A] (M\i) at (\i,0) {\color{white}A};
}

\foreach \i/\j in {0/1,1/2,2/3,3/4,4/5}
{
    \draw[bond] (M\i)--(M\j);
}

\node[left] at (-1.8,-1.1) {\small Diblock};

\foreach \i in {0,1,2}
{
    \node[A] (A\i) at (\i,-1.1) {\color{white}A};
}

\foreach \i in {3,4,5}
{
    \node[B] (B\i) at (\i,-1.1) {\color{white}B};
}

\foreach \i/\j in {0/1,1/2}
\draw[bond] (A\i)--(A\j);

\draw[bond] (A2)--(B3);

\foreach \i/\j in {3/4,4/5}
\draw[bbond] (B\i)--(B\j);

\end{scope}

\draw[line width=1.3pt] (4.1,0.3) -- (4.1,-1.5);
\draw[line width=1.3pt] (4.17,0.3) -- (4.17,-1.5);
\begin{scope}[scale=0.73, shift={(10.8,2.2)}]


\node[left] at (-1.8,-2.2) {\small Triblock};

\foreach \row in {0,1}
{
    \pgfmathsetmacro{\y}{-2.4-0.9*\row}

    \foreach \i in {0,1,2}
    \node[A] (TA\row\i) at (\i,\y) {\color{white}A};

    \foreach \i in {3,4,5}
    \node[B] (TB\row\i) at (\i,\y) {\color{white}B};

    \foreach \i/\j in {0/1,1/2}
    \draw[bond] (TA\row\i)--(TA\row\j);

    \draw[bond] (TA\row2)--(TB\row3);

    \foreach \i/\j in {3/4,4/5}
    \draw[bbond] (TB\row\i)--(TB\row\j);
}

\draw[bbond] (TB05)--(TB15);


\node[left] at (-1.8,-4) {\small Alternating};

\node[A] (AltA1) at (0,-4.4) {\color{white}A};
\node[B] (AltB1) at (1,-4.4) {\color{white}B};
\node[A] (AltA2) at (2,-4.4) {\color{white}A};
\node[B] (AltB2) at (3,-4.4) {\color{white}B};
\node[A] (AltA3) at (4,-4.4) {\color{white}A};
\node[B] (AltB3) at (5,-4.4) {\color{white}B};

\draw[bond]  (AltA1)--(AltB1);
\draw[bbond] (AltB1)--(AltA2);
\draw[bond]  (AltA2)--(AltB2);
\draw[bbond] (AltB2)--(AltA3);
\draw[bond]  (AltA3)--(AltB3);
\end{scope}
\end{tikzpicture}
\caption{Illustration of polymer chain-blocks.}
\label{fig:block-structures}
\end{figure}
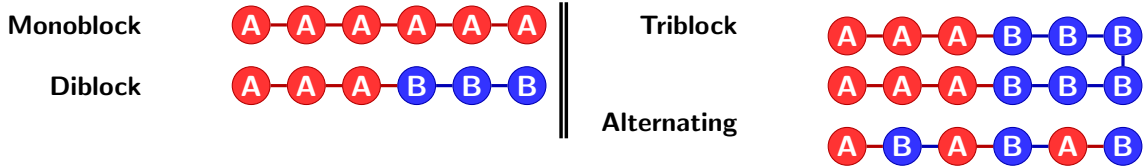
In a diblock copolymer melt A/B, the interface is a critical region where the incompatible A-block and B-block polymers meet. Because of their chemical incompatibility (as discussed  for instance in \cite{ GrSa06,Rus96, Sem85}), the two blocks typically dislike mixing together (much like oil and water). Consequently, they naturally tend to separate at the macroscale, however,  their covalent bonding prevents macroscopic phase separation, leading the two blocks to spontaneously self-assemble at the nanoscale
by forming some periodic nanostructures such as lamellae, cylinders or spheres. The distribution of energy-driven patterns formed at an interface in a screened model of diblock copolymer melts is strongly influenced by the geometry of the interface. Accordingly, the problem of identifying domains of uniform energy distribution has attracted significant attention (see, e.g., \cite{OhKa86,Rus96,Sem85}). See also \cite{Mul12} for a statistical physics analysis of geometry-controlled domain interfaces in block copolymer films.  Furthermore, experimental evidences in  \cite{GrSa06} highlight  the importance periodic undulated cylindrical and spherical interfaces  for diblock copolymers.
\smallskip

\par
In this paper, we  are interested in the study of geometry of periodic interface  in $\mathbb{R}^3$, in which the energy  driven pattern created by diblock copolymer is uniformly distributed.
We focus our investigation on a class of periodic stationary sets in $\R^3$ for  the following  sharp functional energy driven pattern in the \textit{screened} model \cite{OhKa86} proposed by Otha-Kawasaki,
\begin{align}
\label{eq:Geom-pblem-interface-like-s-perim}
\cP_\gamma(\Omega):=  |\partial\Omega|+{\gamma}   \int_{\Omega}\int_{\Omega}{G}_{ \kappa}(|x-y|)  \mathrm{d}x\mathrm{d}y,
\end{align}
where
 $G_\kappa(r)=\frac{1}{r} e^{-\kappa r}$ is the repulsive Yukawa potential (also often referred to as  the repulsive \textit{screened} Coulomb potential), $\gamma>0$ is the regulation parameter encoding the strength of the repulsive interaction and $\kappa>0$
 is the screening factor.
 The Yukawa potential \cite{Yuk47} $\overline{G}_\kappa(x) = G_\kappa(|x|)$, which encodes the nonlocality in the model \eqref{eq:Geom-pblem-interface-like-s-perim}, is the Green's function of $(-\Delta + \kappa^2)$ meaning that
 \begin{align*}
-\Delta \overline{G}_\kappa + \kappa^2 \overline{G}_\kappa = \delta_0\qquad\text{ in $\R^3$}.
\end{align*}
This potential plays an important role in the theory of elementary particles (see \cite{Yuk47}). On the other hand, $|\partial\Omega|$ is the surface area enclosing $\Omega\subset \R^3$ encodes the locality in \eqref{eq:Geom-pblem-interface-like-s-perim} and depends directly on its mean curvature.
Note that, the model is termed ``\textit{screened}" because $\kappa > 0$, as,  long-range (nonlocal) interactions are exponentially damped (or screened) due the factor $e^{-\kappa r}$, by making long-range interactions less feasible. In contrast, the unscreened model ($\kappa = 0$) is driven  by the usual Newtonian or Coulomb interaction kernel $G_0(|x|) =|x|^{-1}$, allowing long-range interactions to occur freely. The latter model is often referred to as Gamow's liquid drop model, introduced in the seminal works \cite{Boh36,Gam30}.
\smallskip

We are interested in finding a class of periodic stationary sets in $\mathbb{R}^3$ for the model \eqref{eq:Geom-pblem-interface-like-s-perim}.  It is proved in \cite{Mur02} that, $\cP_\gamma$ provides the sharp interface limit of the Otha-Kawasaki model and for some other similar diffused models describing pattern formation. For this energy  $\cP_\gamma$, comprising   short-range interaction   (surface tension) and  attractive Yukawa potential, a stationary set $\Omega$ may separate into well ordered disjoint components.  This separation give rise to patterns    in the limit of    some reaction-diffusion model with Coulomb-type  attractive interaction such as the  Ohta-Kawasaki model of block polymers and the FitzHugh-Nagumo reaction-diffusion system, see for  instance  \cite{GMP96,SeAn95} and the figures therein.
Indeed,  the Ohta-Kawasaki model, is
a canonical mathematical model in the studies of energy-driven pattern forming @systems. This model, originally proposed in \cite{OhKa86}   investigates the statistical properties of the microphase separation in the strong-segregation limit and describes different morphologies
observed in diblock copolymer melts (for example see also \cite{BaFr99,Sem85}). The energy functional  is defined   by
\begin{align*}
\cE_\varepsilon(u):=\int_{E} \Big(\frac{\varepsilon^2}{2}|\nabla u|^2+W(u)\Big)\mathrm{d}x +\frac{\sigma}{2}\int_{E}\int_{E}G(x,y)(u(x)-\overline{u})(u(y)-\overline{u})\mathrm{d}x\mathrm{d}y,
\end{align*}
for an open set $E\subset\R^3$,
where $u:E\to \R$ and
\begin{align}
\label{eq:Neut-cond}
\overline{u}=\frac{1}{|E|}\int_{E}u(x)\mathrm{d}x\in (-1,1)
\end{align}
is the neutrality charge.
Here $G$ is the Green function of $E $ with zero Neumann boundary condition and  $W$ is  a double well potential with global minima at $u=\pm1$. For the precise values of the constants $\varepsilon$ and $\sigma$ we refer the reader to \cite{ChRe03} and  \cite{NiOh95}. In a diblock copolymer, we have two different monomer bonded together, which can be represented by $u = -1$ and $u = 1$ respectively. The connectivity of the two different monomers are due to the long-range (Green function $G$) interaction while the double well potential $W$ intends  to  segregate the  monomers. Moreover  $ \frac{\varepsilon^2}{2}|\nabla u|^2$,  interfacial energy density,  tends to   force the polymers to be locally uniformly distributed.  Periodic patterns are in general   expected  to minimize the  Ohta-Kawasaki energy in the
macroscopic setting, see \cite{OhKa86}. The first rigorous proof of this fact is due to    M\"{u}ller \cite{Mul93} in 1-dimension. Up to normalization of constants, the  Ohta-Kawasaki functional $\cE_\varepsilon$ admits       the following   ``refined version'' (or geometric model)
\begin{align}
\label{eq:GeomModel-BV}
\overline{\cE}_\varepsilon(u):=\frac{\varepsilon}{2}   \int_{E}|\nabla u |+\frac{\widetilde{\sigma} }{2}   \int_{E}\int_{E}\widetilde{G}_{ \widetilde{\kappa}}(x,y)(u(x)-\overline{u})(u(y)-\overline{u})\mathrm{d}x\mathrm{d}y ,
\end{align}
where   $u:E\to \{-1,1\}$, $\widetilde{\sigma}>0$,   $\widetilde{\kappa}=\frac{1}{\sqrt{W''(1)}}$ and  $\widetilde{G}_{ \widetilde{\kappa}}(x,\cdot)$ solves
$$-\Delta \widetilde{G}_{ \widetilde{\kappa}}(x,y)+\kappa^2 \widetilde{G}_{ \widetilde{\kappa}}(x,y)=\delta_x(y),$$ with Neumann boundary condition on $\partial E$.  We recall that for this model, the neutrality condition \eqref{eq:Neut-cond} is no longer imposed.  The paper of Muratov \cite{Mur02} contains a rigorous  derivation of  \eqref{eq:GeomModel-BV} from the general   mean-field free energy functional. Equilibrium patterns (spots, stripes, and the annuli) an morphlogical instabilities have been also studied \cite{Mur02}.
Muratov   showed in \cite{Mur10} that, in two space dimensions,  the minima  of  $\cE_\varepsilon$ and $\overline{\cE}_\varepsilon$ scales similarly for $\varepsilon$ is small.
Moreover, he proved that at certain energy level,  after suitable normalization and scaling  background charge density ($\overline{u}= \overline{u}_\varepsilon\to -1$, as $\varepsilon\to0$), the sharp interface energy $\overline{\cE}_\varepsilon$ is minimized by droplets  uniformly distributed throughout the domain $E$ when $\varepsilon\to 0$.     See also the work of  Goldmann, Muratov and Serfaty  in \cite{GMS13,GMS14}  for  related variant of   $\Gamma$-convergence on the 2-dimensional torus. We should notice that if $\overline{u}\in (-1,1)$ is fixed independent on $\varepsilon$ and $E$ the square lattice, Alberti, Choksi, and Otto in \cite{ACO09}   proved the $\Gamma$-convergence of  $\cE_\varepsilon$  to the unscreened ($\kappa=0$) version of \eqref{eq:GeomModel-BV}.

We use local inversion arguments to build Delaunay-type    periodic patterns in $\R^3$.  In the case $\kappa=0$, Ren and Wei constructed  several interesting domain patterns in their series of papers \cite{ReWe05,ReWe03,ReWe03ABC,ReWe07,ReWe07single,ReWe08,ReWe09,ReWe11,ReWe13,ReWe14,ReWe14Double} of type lamellar,
ring/cylindrical, and spot/spherical solutions.
Going back to the model \eqref{eq:GeomModel-BV}, since the Yukawa potential $G_\kappa$ is integrable on $\R^3$, we consider  $E=\R^3$ then $\widetilde{G}_{\widetilde{\kappa}}(x,y)=G_{\widetilde{\kappa}}(|x-y|)$ and we look for domain patterns $\overline{\Omega}$ which are periodic in different directions with (possibly) different periods.
Assuming that   the domain pattern $\overline{\Omega}$ that we are interested in  consists of the monomer $u=+1$  surrounded by the sea of monomers $u=-1$ in  the whole space $E=\R^3$ we can therefore let $\overline{u}=-1$. Hence the whole configuration is given by   $u=\tau_{\overline{\Omega}}=1_{\overline{\Omega}}-1_{\R^3\setminus {\overline{\Omega}}}$. Plugging these in   \eqref{eq:GeomModel-BV}, we then have
\begin{align*}
\overline{\cE}_\varepsilon(\tau_{\overline{\Omega}}):=\varepsilon |\partial \overline{\Omega}|+\widetilde{\sigma }   \int_{\overline{\Omega}}\int_{\overline{\Omega}}{G}_{ \widetilde{\kappa}}(|x-y|)  \mathrm{d}x\mathrm{d}y.
\end{align*}
Since we are looking at large domains, letting $\overline{\Omega}=\varepsilon {\Omega}$,  by a change of variable, we have
\begin{align*}
\overline{\cE}_\varepsilon(\tau_{\overline{\Omega}}):=\varepsilon^3
\Big(  |\partial\Omega|+ \varepsilon^2  \widetilde{\sigma }   \int_{\Omega}\int_{\Omega}{G}_{\varepsilon \widetilde{\kappa}}(|x-y|)  \mathrm{d}x\mathrm{d}y \Big).
\end{align*}
Note the expression in the parenthesis was derived by Muratov in \cite{Mur02}. We are therefore led to our initial problem \eqref{eq:Geom-pblem-interface-like-s-perim}, with $  \gamma= \varepsilon^2  \widetilde{\sigma }   $ and $\kappa= \varepsilon \widetilde{\kappa}$.
It is known, see for instance \cite{Mur02, FFMMM15},  that  a stationary set (or critical point) $\Omega$ for the functional $\cP_\gamma$ under volume constraints satisfy the following Euler-Lagrange equation
\begin{align}
\label{eq:main-problem}
H_{\partial\Omega}(x)+\gamma \int_{\Omega}G_\kappa(|x-y|) \mathrm{d}y=\mathrm{Const}. \qquad \textrm{ for all } x\in \partial\Omega,
\end{align}
where $H_{\partial\Omega}$ is the  mean curvature (positive for a ball) of $\partial\Omega$. Here and in the following,
for every set $\Omega$ with $C^2$ boundary (not necessarily bounded), we define
$\cH_\Omega:\partial\Omega\to \R$ with
\begin{align}
\label{eq:def-cH}
\cH_{\Omega}(x)=  H_{\partial\Omega}(x)+\gamma \int_{\Omega}G_\kappa(|x-y|) \mathrm{d}y,\qquad \textrm{ for  $x\in \partial\Omega$}.
\end{align}
In this paper, \textit{domain patterns} are sets $\Omega\subset\R^3$ for which   $\cH_\Omega\equiv \mathrm{const}$ on $\partial\Omega$.  Namely, $\partial\Omega$ is an interface of the sharp Ohta-Kawasaki model of polymer. We are interested in  domains with boundary of Delaunay-type
surface in which the problem \eqref{eq:main-problem} is solvable. More precisely, we are interested in domains of revolution of the form
\begin{align*}
\Omega_\varphi:=\left\{ (t, z)\in \R\times
\mathbb{R}^{2} \,:\, |z|<\varphi(t)  \right\}
\end{align*}
where  $\varphi: \R \to (0,\infty)$ is an even, $2\pi$-periodic and  smooth function. Recall that, a Delaunay surface is a surface of revolution $\partial \Omega_\varphi$ with constant mean curvature. Consequently, when the nonlocal Yukawa interaction is completely dropped out ($\gamma = 0$), the domains $\Omega_\varphi$ enclosed by Delaunay surfaces $\partial \Omega_\varphi$ are solutions to the variational problem \eqref{eq:main-problem} (or stationary interfaces of the energy \eqref{eq:Geom-pblem-interface-like-s-perim}).
Here is our main result.
\begin{theorem}\label{thm:main-thm}
For each $\kappa>0$, there exists $\gamma_*=\gamma_*(\kappa)\in  \Big(0, \, \frac{1}{2\pi} \tfrac{1}{\frac{1}{\kappa} - \frac{1}{\sqrt{\kappa^2+1}}}\Big)$ such that for every $\alpha \in (0,1)$ and $\gamma\in (0,\gamma_*)$, there exists   $\lambda_{*}=\lambda_*(\kappa, \gamma)>0$ and a smooth curve
\begin{align*}
 I: (-\varepsilon_0,\varepsilon_0) &\longrightarrow  (0,\infty) \times C^{2,\alpha}(\R),\quad  s\longmapsto (\lambda_s,u_s)
 \end{align*}
satisfying   $I(0)= (\lambda_0,u_0) \equiv( \lambda_{*},0)$. Moreover,  the  following properties hold.
\begin{enumerate}[$(a)$]
\item For each $s \in (-\varepsilon_0, \varepsilon_0)$,  $u_s\in C^{2,\alpha}(\R)$ is even, $2\pi$-periodic, and takes the form
\begin{align*}
u_s(t) = s\big(\cos(t) + \omega_s(t)\big),
\end{align*}
where $\omega: (-\varepsilon_0, \varepsilon_0) \to C^{2,\alpha}(\mathbb{R})$, $s \mapsto \omega_s$ is a smooth curve satisfying $\omega_0 \equiv 0$, $\omega_s\in C^{2,\alpha}(\R)$ is even, $\omega_s$ is $2\pi$-periodic and satisfies
\begin{align*}
\int_{0}^{2\pi} \omega_s(t) \cos(t) \mathrm{d}t = 0.
\end{align*}
\item Each function $\varphi_s = \lambda_s + u_s$ satisfies  $\varphi_s\in C^{2,\alpha}(\R)$,  $\varphi_s\geq \frac{\lambda_*}{2}$, $\varphi_0=\lambda_*$ and
\begin{align}
\label{eq:nmc-const-mt}
H_{\partial\Omega_{\varphi_s}}(x) + \gamma \int_{\Omega_{\varphi_s}} G_\kappa(|x-y|)  \mathrm{d}y = C_{\lambda_s,\kappa} \qquad \text{for all } x \in \partial\Omega_{\varphi_s}
\end{align}
where the constant $C_{\lambda_s,\kappa}$ satisfies $C_{\lambda_s,\kappa}\leq C_{\lambda_*,\kappa}$ with
\begin{align*}
C_{\lambda_s,\kappa} &=H_{\partial\Omega_{\lambda_s}}\hspace{-0.7ex}
(0,\lambda_s e_1)+\gamma \int_{\Omega_{\lambda_s}}\hspace{-2ex} G_\kappa(|(0,\lambda_s e_1)-z|) \mathrm{d}z=\frac{1}{\lambda_s}
+ 2\gamma\lambda_s^2 \int_{B_1} \hspace{-2ex}K_0(\lambda_s \kappa|e_1-y|)\mathrm{d}y,
\end{align*}
\end{enumerate}
where $e_1=(1,0)\in \R^2$ and  $K_\nu$ is the modified Bessel function of the second kind of order $\nu$.
\end{theorem}
In other words,  the unbounded periodic Delaunay-type domains $(\Omega_{\varphi_s})_{s \in (-\varepsilon_0, \varepsilon_0)}$ are stationary sets for the energy functional \eqref{eq:Geom-pblem-interface-like-s-perim} (or equivalently, solutions to the geometric variational problem \eqref{eq:main-problem}). The flat cylinder $\Omega_{\lambda_*}$ provides a particular solution to the geometric variational problem \eqref{eq:main-problem} with constant mean curvature. In fact, all flat cylinders $\Omega_{\lambda}$ with $\lambda > 0$ are also trivial solutions with constant mean curvature to problem \eqref{eq:main-problem}; see Proposition \ref{prop:var-functionals} below.   Note that while equation \eqref{eq:nmc-const-mt} can be rewritten as $\cH_{\Omega_{\lambda_s+u_s}} = \cH_{\Omega_{\lambda_s}}$, this relation allows us to analyze the linearization of $\varphi \mapsto \cH_{\Omega_\varphi}$ around flat cylinders $\Omega_{\lambda}$, $\lambda>0$ and apply the Crandall--Rabinowitz bifurcation theorem to establish the existence of $\varphi_s$ and $\lambda_s$.

Let us mention some works related to ours.  In the absence of nonlocal interaction  that is ($\gamma=0$), the original landmark  paper by C. Delaunay \cite{Del41} classifies all Delaunay surfaces, that is, all solutions to \eqref{eq:main-problem} as surfaces of the form $\Omega_{\varphi}$ generated by revolving the roulettes of conic sections. These surfaces includes six family types: planes, cylinders, spheres, catenoids, unduloids, and nodoids. The paper \cite{Fall18} constructs nontrivial unbounded equilibrium pattern sets of $\cP_\gamma$, with $\frac{1}{2\pi}\frac{1}{\frac{1}{\kappa}-\frac{1}{\sqrt{\kappa^2+1}}} < \gamma <
\frac{1}{2\pi}\frac{1}{\frac{1}{(\kappa^2+1)^{3/2}}}$, that are multiply periodic and bifurcate from lattices of round spheres, straight cylinders, and slabs.
The existence of Delaunay-type periodic hypersurfaces with constant nonlocal fractional mean curvature is established in \cite{CFSW18}; see also \cite{MNT20}, where an analogous study is carried out. More recently, in the context of  in Gamow's liquid drop model, i.e., $\kappa=0$, the existence of Delaunay-like compact equilibria is established in \cite{dPMZ25}. A detailed survey of results examining the connection between Gamow liquid drop models, nonlocal energy minimizers, and diblock copolymer interfaces can be found in \cite{CMT17}. Recent findings by \cite{BMP25} establish the existence of non-spherical minimizers in generalized liquid drop models featuring the truncated (or abruptly screened) Riesz interaction kernel $G(|x|) = \mathds{1}_{\{|x| < \kappa\}} |x|^{-\alpha}$ or the screened Riesz-Yukawa interaction kernel $G(|x|) = e^{-|x|/\kappa} |x|^{-\alpha}$, where $x \in \mathbb{R}^n \setminus \{0\}$ and $\alpha \in (0, n)$.

\par The remainder of this paper is organized as follows. In Section \ref{sec:preliminaries} we recast the problem for pattern-domain of Delaunay-type  $\Omega_\varphi$ driven by H\"older functions $\varphi$. In Section \ref{sec:spectral-decomp}, we analyze the spectrum of the linear operator $L_\lambda = D\cF(\lambda): C^{2,\alpha}(\R)\to C^{0,\alpha}(\R)$, which the Fr\'echet derivative at $\varphi=\lambda$ of the operator $\varphi\mapsto \cF(\varphi) = \cH_{\Omega_\varphi}$, $\varphi>0$. Section \ref{sec:main-result} is devoted to proving Theorem \ref{thm:main-thm}, and Section \ref{sec:different-F} establishes the Fr\'echet $C^\infty$-differentiability of the functional $\cF$.

\medskip
\textit{Acknowledgements.}
The main topic of this work stems from the thesis of
G. Foghem, completed in 2016 under the supervision of M. M. Fall at the African Institute for Mathematical Sciences (AIMS) Cameroon. The authors express their sincere gratitude to AIMS-Cameroon for its hospitality and supportive research environment during the execution of this project.

\section{Preliminaries and Notations}
\label{sec:preliminaries}
%
\subsection{Notations}
Throughout, the notation $|\cdot|$ may denote the Euclidean norm in $\R^2$ or $\R^3$, the Lebesgue measure, or the $(d-1)$-dimensional Hausdorff measure, depending on the context. We shall adopt the following notations.

Given two functions $f,g:\R\to \R$, the notation $f(x)\sim g(x)$ as $x\to x_0$ means $\lim_{x\to x_0}\frac{f(x)}{g(x)}=1$ while by $f(x)\approx g(x)$ as $x\to x_0$, we mean $\lim_{x\to x_0}\frac{f(x)}{g(x)}=c $, for some positive constant $c$.
By $f(x) = O(g(x))$ as $x \to x_0$, we mean that $|f(x)| \leq C|g(x)|$ for some constant $C > 0$ in a neighborhood of $x_0$.

\subsection{Pattern-domains of  Delaunay-type}
In this section, we recast the variational problem \eqref{eq:main-problem} unbounded pattern-domains  whose boundaries are Delaunay-type surfaces driven by a periodic H\"older function.  We focus our investigation on the class of Delaunay-type surfaces generated by the rotation of a \textit{periodic curve}, that is, domains of the form
\begin{align*}
\Omega_\varphi
:=\{ (t, z)\in
\R\times \R^2 \,:\, |z|<\varphi(t)  \}
= \{ (t, \varphi(t) z)   \,:\, t\in \mathbb{R}, ~~z\in B_1 \} \subset\mathbb{R}^3,
\end{align*}
where  $\varphi: \mathbb{R} \to (0,\infty)$ is an even and $2\pi$-periodic function and $B_1$ is the unit ball of $\mathbb{R}^2$. The domain $\Omega_\varphi$ is clearly unbounded periodic domain obtained by rotation of the curve $t\mapsto \varphi(t)$
around the $t$-axis.  It clearly appears that, the boundary $\partial\Omega_\varphi$ is given by
\begin{align}\label{eq:boundary-varphi}
\partial{\Omega_\varphi}=S:= \{(t,x,y)\in\mathbb{R}^3: F(t,x,y):= x^2+y^2-\varphi^2(t) = 0  \}.
\end{align}
The existence of such domain gives a solution of our main equation \eqref{eq:main-problem}.   The mean curvature $H_S$  of a smooth surface $S = \{ (x_1,x_2,x_3) \in \mathbb{R}^3 \,:\, F(x_1,x_2,x_3)=0 \}$ defined by a smooth function $F$, see for instance \cite{Gold05}, is given by
\begin{align}\label{eq:mean-curvature}
H_S =\tau \mbox{div}\Big(\frac{\nabla F}{|\nabla F|}\Big)=\tau \frac{|\nabla F|^2\, \text{Trace}(\mbox{Hess}(F))-\nabla F  \mbox{Hess}(F)  \nabla F^{\mathsf {T}} } { |\nabla F|^3},
\end{align}
where the factor $\tau\in \R$ is to be chosen,   $\nabla F=(F_{x_1}, F_{x_2}, F_{x_3})$ and $\mbox{Hess(F)}= (F_{x_i x_j})_{1\leq i,j\leq 3}$ are respectively the gradient and the Hessian matrix of $F$. Clearly, the sign of the mean curvature depends on the choice of $\pm F$, which corresponds to the orientation of the chosen unit normal vector.
Here, we choose the orientation such that the mean curvature of the unit ball satisfies $H_{\mathbb{S}^2} = 2$, with $ \mathbb{S}^2$ is the unit sphere of $\R^3$, namely $\tau=1$.

\begin{proposition}\label{prop:var-functionals}
 Let $ e_1= (1,0)$ and $\gamma > 0$.  For a $C^2$ function $\varphi: \mathbb{R} \to \mathbb{R} \setminus \{0\}$, the map $$\varphi \mapsto \cF(\varphi) := \cH_{\Omega_\varphi}$$
is defined for every $(t, \varphi(t)z) \in \partial \Omega_\varphi$ with $z \in \mathbb{S}^1 = \partial B_1$, by
\begin{align*}
 \cF(\varphi)(t)= \cH_{\Omega_\varphi} ((t, \varphi(t) z))= \cF_0(\varphi)(t)+ \gamma \cF_1(\varphi)(t)
\end{align*}
where the Yukawa interaction takes the form,
\begin{align}\label{eq:Yukawa-term}
\int_{\Omega_\varphi} \hspace{-1.5ex}G_\kappa(|(t, \varphi(t) z)  -\xi|)\mathrm{d}\xi =\hspace{-1.5ex}\int_\mathbb{R}\int_{B_1} \hspace{-1.5ex} \varphi^2(s) G_\kappa(|(t, \varphi(t)e_1)-(s,\varphi(s)y)|)\mathrm{d} y \mathrm{d}s=:\mathcal{F}_1(\varphi)(t).
\end{align}
While, the mean curvature of $\partial\Omega_\varphi$  at $(t, \varphi(t) z)  \in\partial \Omega_\varphi$ is given by
\begin{align}\label{eq:Mean-Curv}
\begin{split}
H_{\partial\Omega_\varphi}((t, \varphi(t)z) )
&=\frac{-\varphi''}{(1+\varphi'^2)^{\frac{3}{2}}}+\frac{1}{\varphi(1+\varphi'^2)^{\frac{1}{2}}}. =: \mathcal{F}_0(\varphi)(t).
\end{split}
\end{align}
Moreover,  for $\lambda \in (0,\infty)$ the function $t\mapsto \cF(\lambda)(t) = C_{\lambda,\kappa}$ is constant, where
\begin{align*}
C_{\lambda,\kappa} &=H_{\partial\Omega_{\lambda}}((0,\lambda e_1))+\gamma \int_{\Omega_{\lambda}}G_\kappa(|(0,\lambda e_1)-z|) \mathrm{d}z=\frac{1}{\lambda} + 2 \gamma\lambda^2 \int_{B_1} K_0(\lambda\kappa|e_1-y|)\mathrm{d}y.
\end{align*}
\end{proposition}
\begin{proof}
Let $\xi = (s, z')\in\Omega_\varphi $ then, $z'\in \varphi(s) B_1$.
Fixing $s\in \R$ and  making   the change of variable $z'=\varphi(s)y' $ with, $y' \in B_1$ so that  $ \mathrm{d}z'=\varphi^2(s)dy'$ implies
\begin{align*}
\int_{\Omega_\varphi} G_\kappa(|(t, \varphi(t) x)  -\xi|)\mathrm{d}\xi
&= \int_\mathbb{R}\int_{\varphi (s)B_1}G_\kappa(|(t, \varphi(t)z)-(s,z')|) \mathrm{d}z'\mathrm{d}s \\
&= \int_\mathbb{R}\int_{B_1}\varphi(s)^2 G_\kappa(|(t, \varphi(t)x)-(s,\varphi(s)y')|)\mathrm{d}y'\mathrm{d}s.
\end{align*}
Since $|e_1|=|z|=1$, there is an orthogonal transformation
$U:\R^2\to \R^2$ such that $z=Ue_1$. Note that  $UU^\mathsf{T}=I_2$  and $|Uw|= |w|$, $w\in \R^2$.
Therefore, letting $y'=Uy$ we get
\begin{align*}
|(t, \varphi(t)z)-(s,\varphi(s)y')|^2 &=(t-s)^2 +| \varphi(t)Ue_1-\varphi(s)Uy|^2
=|(t, \varphi(t)e_1)-(s,\varphi(s)y)|^2.
\end{align*}
The change of variable  $y'=Uy$ with $\mathrm{d}y'= |\det U|\mathrm{d}y= \mathrm{d}y$ implies
\begin{align*}
\int_{\Omega_\varphi} G_\kappa(|(t, \varphi(t) x)  -\xi|)
\mathrm{d}\xi
&= \int_{\mathbb{R}} \int_{B_1}\varphi^2(s)
G_\kappa(|(t, \varphi(t)\mathrm{U}e_1)-(s,\varphi(s)y')|)\mathrm{d}y'\mathrm{d}s\\
&=   \int_\mathbb{R}\int_{B_1} \varphi^2(s) G_\kappa(|(t, \varphi(t)e_1)-(s,\varphi(s)y)|)\mathrm{d}y\mathrm{d}s.
\end{align*}
Now we compute the mean curvature of $
\partial{\Omega_\varphi}$ given in \eqref{eq:boundary-varphi}. A straightforward computation yields $\nabla F(t,x,y)=\big( -2\varphi(t)\varphi'(t),  2x, 2y\big)$ and the Hessian matrix
\begin{align*}
\text{Hess}(F)(t,x,y) =
\begin{pmatrix}
-2(\varphi\varphi')'(t) & 0 & 0 \\
0 & 2 & 0 \\
0 & 0 & 2
\end{pmatrix},
\end{align*}
where $(\varphi\varphi')' = \varphi'^2 + \varphi\varphi''$. We deduce that
\begin{align*}
\text{Trace}(\text{Hess}(F)) &=\Delta F= 4 - 2\big(\varphi'^2 + \varphi\varphi''\big),\\
|\nabla F|^2 &= 4(x^2+y^2+\varphi^2\varphi'^2 ) = 4\varphi^2(1+\varphi'^2 ), \\
\nabla F \,\text{Hess}(F) \,\nabla F^{\mathsf{T}}
&= 8(x^2+y^2 - \varphi^2\varphi'^2 (\varphi'^2 +\varphi\varphi'')) = 8\varphi^2(1-\varphi'^2 (\varphi'^2 +\varphi\varphi'')).
\end{align*}
Substituting all into the mean curvature formula \eqref{eq:mean-curvature} with $\tau=1$ yields
\begin{align*}
H_{\partial\Omega_\varphi}((t, \varphi(t)z) )
=\frac{-\varphi''}{(1+\varphi'^2)^{\frac{3}{2}}}+ \frac{1}{\varphi(1+\varphi'^2)^{\frac{1}{2}}}.
\end{align*}
 Now for $\lambda\in \R$, we clearly have  $\mathcal{F}_0(\lambda)(t) = \frac{1}{\lambda}$ while letting $r= t-s$  yields
\begin{align*}
\cF_1(\lambda)(t)= \lambda^2 \int_{\mathbb{R}}\int_{B_1 }G_\kappa((r^2+\lambda^2| e_1-y|^2)^{\frac{1}{2}}) \, \mathrm{d}y\, \mathrm{d}r =\mathrm{const}.
\end{align*}
So that  using the relation  \eqref{eq:Bessel4-bis} below we get
\begin{align*}
C_{\lambda,\kappa} 
%
%
=\frac{1}{\lambda} + 2 \gamma\lambda^2 \int_{B_1}\int_{0}^{\infty} \frac{e^{-\kappa\sqrt{r^2+\lambda^2| e_1-y|^2}}}{\sqrt{r^2+\lambda^2| e_1-y|^2}}\mathrm{d}r \mathrm{d}y
=\frac{1}{\lambda} + 2 \gamma\lambda^2 \int_{B_1} K_0(\lambda\kappa|e_1-y|)\mathrm{d}y.
\end{align*}
\end{proof}

\begin{remark}[New problem to solve]
\label{rem:new-problem}
Since  $t\mapsto\cF(\lambda)(t)$  is constant for every $\lambda>0$ with $\cF= \cF_0+\gamma\cF_1$,  in view of \eqref{eq:Mean-Curv} and \eqref{eq:Yukawa-term}, to solve the  Euler-Lagrange equation \eqref{eq:main-problem}  associated with the Delaunay-type surface $\Omega_\varphi $, it  is sufficient to show the existence of a function $u: \R\to \R$  and $\lambda>0$ satisfying
\begin{align}\label{eq:new-equation-to solve}
\cF(\lambda+u)=\cF(\lambda),
\end{align}
and take $\varphi=\lambda+ u$ afterwards.
\end{remark}

\noindent The cornerstone of our strategy for the variational problem \eqref{eq:new-equation-to solve} is the Crandall-Rabinowitz bifurcation theorem \cite[Theorem 1.7]{CrRa71}, which we recall here for reader convenience.
\begin{theorem}[{Crandall-Rabinowitz: \cite[Theorem 1.7]{CrRa71}}] \label{thm:crandall}
Let X, Y be Banach spaces, V a neighborhood of 0 in X and $F:(-1,1)\times V\to  Y$
have the properties :
\begin{enumerate}[$(a)$]
\item $F(t,0) = 0 $  for $|t|<1$.
\item The partial derivatives $F_t$, $F_z$ and $F_{tx}$ exist and are continuous,
\item $\ker(F_x(0, 0))$ and $Y\setminus Range(F_x(0, 0))$ are one-dimensional,
\item $F_{tx}(0,0)x_0\notin Range(F_x(0, 0))$, where
\begin{align*}
\ker(F_x(0,0))= \textrm{span}\{x_0\}.
\end{align*}
\end{enumerate}
Let $Z$ be any complement of $\ker(F_x(0, 0))$ in $X$. Then there is a neighborhood $U$ of $(0,0)$ in $\mathbb{R}\times X$, an interval $(-a,a),$ and continuous functions $\varphi: (-a,a) \to  \mathbb{R} $, $\psi: (-a,a) \to  Z $ such that $\varphi(0) =0$, $\psi(0)=0$ and
\begin{align}\label{eq:crandall}
F^{-1}(\{0\})\cap U = \lbrace (\varphi(\alpha), \alpha x_0+\alpha\psi(\alpha)): |\alpha|<a\rbrace\cup \lbrace(t,0): (t, 0)\in U\rbrace.
\end{align}
In other word for $(t,x)\in U$ $F(t,x)= 0$  if and only if $ x=0$  or there is $\alpha\in (-a,a)$ such that
\begin{align*}
t=\varphi(\alpha)
\quad \text{and}\quad
x= \alpha x_0+\alpha\psi(\alpha).
\end{align*}
Moreover, if $F_{xx}$ is also continuous, the functions $\varphi$ and $\psi$ are continuously differentiable.
\end{theorem}
\subsection{Set-up in H\"older spaces}
Our approach to solve the equation \eqref{eq:new-equation-to solve},  requires appropriate Banach space wherein the functional $\varphi \mapsto \cF(\varphi)=\cH_{\Omega_\varphi}$ is Fr\'echet differentiable.
A natural choice for the variational space of the variable $\varphi$ is the H\"older space $C^{\ell,\alpha}(\mathbb{R})$ with
 $$C^{\ell,\alpha}(\mathbb{R})= \{ u\in C(\R^d)\;:\, \|u\|_{C^{\ell;\alpha}(\R)}<\infty\},$$
 is a  Banach space, see for instance \cite{AdFo03}, equip with the norm
\begin{align*}
\|u\|_{C^{\ell,\alpha}(\mathbb{R})} : = \sum_{j=0}^\ell \|u^{(j)}\|_{L^\infty(\mathbb{R})} +  \sup_{\stackrel{s,t \in \mathbb{R}}{s \not = t}}\frac{|u^{(\ell)}(s)-u^{(\ell)}(t)|}{|s-t|^{\alpha}},
\end{align*}
where $\ell \in \mathbb{N}$ and $\alpha \in [0,1]$. The space $C^{\ell,0}(\mathbb{R})$ is endowed with the norm
\begin{align*}
\|u\|_{C^{\ell,0}(\mathbb{R})} : = \sum_{j=0}^\ell \|u^{(j)}\|_{L^\infty(\mathbb{R})}  .
\end{align*}
 It is also  legitimately temping to consider a Sobolev space $W^{m,p}(\R)$ as a natural framework; however, Morrey's embeddings imply that one-dimensional Sobolev spaces--except in the case $p=m=1$--are continuously embedded in H\"older spaces.

Now, for $\alpha\in (1,0)$ we consider the Banach spaces
\begin{align*}
X:= \bigl\{  \varphi \in C^{2,\alpha}(\mathbb{R}) \quad \textrm{is even and $2\pi$-periodic }\big\}, \\
Y:= \bigl\{  \varphi \in C^{0,\alpha}(\mathbb{R}) \quad \textrm{is even and $2\pi$-periodic }\big\},
\end{align*}
equip with the norms $\|\cdot\|_X=\|\cdot\|_{C^{2,\alpha}(\R)}$ and $\|\cdot\|_Y=\|\cdot\|_{C^{0,\alpha}(\R)}$.
It is not difficult to verify that, the periodicity implies that the set $\mathcal{O}\cap X$ with
\begin{align*}
\mathcal{O}:=\{ \varphi \in C^{2,\alpha}(\mathbb{R})\,:\, \varphi>0 \}
\end{align*}
is an open subset of $X$viewed as a Banach space on its own.

\begin{remark}[$\cF$ is well defined]\label{rem:funct-F} The map $\cF: \cO\cap X \to Y$ with
\begin{align}
\cF(\varphi)
&:=\cH_{\Omega_\varphi}= \cF_0(\varphi)+\gamma\cF_1(\varphi),\\
\label{eq:Func-F0}
\mathcal{F}_0(\varphi)(t)&:= \frac{-\varphi''(t)}{(1+\varphi'^2(t))^{\frac{3}{2}}}+ \frac{1}{\varphi(t)(1+\varphi'^2(t))^{\frac{1}{2}}},
\\
\mathcal{F}_1(\varphi)(t) &:=\int_{\mathbb{R}}\int_{B_1 }\varphi^2(s)G_\kappa(|(t, \varphi(t) e_1)-(s,\varphi(s)y)|) \mathrm{d} y \mathrm{d}s,\label{eq:Func-F1}
\end{align}
is well-defined.  Indeed, it is straightforwards to verify that  $\varphi(-t)=\varphi(t)$  for ever $t\in \R$ then  $\cF(\varphi)(-t)= \cF(\varphi)(t)$  for ever $t\in \R$. In addition, $\varphi$ is $2\pi$-periodic i.e., if  $\varphi(t+2\pi)=\varphi(t)$ for ever $t\in \R$ then  $\cF(\varphi)(t+2\pi)= \cF(\varphi)(t)$ for every $t\in \R$.
On the other hand,   note that the H\"older space $C^{\ell,\alpha}(\R)$ is closed under pointwise product multiplication, that is,  if $f, g \in C^{\ell,\alpha}(\R)$, then $f\cdot g\in C^{\ell,\alpha}(\R)$. Therefore, the condition $\varphi \ge \delta > 0$ ensures that the map $\mathcal{F}_0: \mathcal{O} \cap X \to Y$ is well-defined.
Meanwhile, using Leibniz's rule it is not difficult to  show that $\cF_1(\varphi) \in C^{1,0}(\R)$, which in turn implies that $\cF_1(\varphi) \in C^{0,\alpha}(\R)$. Thus the map $\mathcal{F}_1: \mathcal{O} \cap X \to Y$ is well-defined.
\end{remark}

\noindent The proof of the following theorem establishing the differentiability of $\mathcal{F}$ will be postponed to Section \ref{sec:different-F}, as it is quite long and involves heavy computations.
\begin{theorem}[Smoothness of $\cF$]
\label{thm:diff-cF}
The functional $\cF:\cO\cap X\to Y$ is $C^\infty$. Moreover, the Fr\'echet derivative $D\cF(\varphi): X\to Y$ is given by
\begin{align*}
D\cF(\varphi)(w)(t):
&:=  -\frac{w''}{(1+\varphi'^2)^{3/2}} + \frac{3\varphi''\varphi'w'}{(1+\varphi'^2)^{5/2}} -\frac{\varphi'w'}{\varphi (1+\varphi'^2)^{3/2}} -\frac{w}{\varphi^2(1+\varphi'^2)^{1/2}}\\
&\quad +\gamma\int_\mathbb{R} \int_{\mathbb{S}^1}(w(s)-w(t))   G_\kappa(|(t, \varphi(t) e_1)-(s,\varphi(s)\theta)|) \varphi(s) \, \mathrm{d} \theta \mathrm{d}s\\
&\quad +\gamma \frac{w(t)}{2}\int_\mathbb{R} \int_{\mathbb{S}^1} |\theta-e_1|^2   G_\kappa(|(t, \varphi(t) e_1)-(s,\varphi(s)\theta)|) \varphi(s)\, \mathrm{d} \theta \mathrm{d}s,
\end{align*}
where $\mathbb{S}^1= \partial B_1$ is the unit sphere of $\mathbb{R}^2$.
\end{theorem}
\begin{proof}
This a direct consequence of Lemma \ref{lem:diff-f0} and Corollary \ref{cor:diff-f1}.
\end{proof}

\section{Spectrum of the linear operator \texorpdfstring{$L_\lambda := D\mathcal{F}(\lambda)$}{L\_lambda := D F(lambda)}}
\label{sec:spectral-decomp}
In this section, we analyze the spectrum of the derivative of $\mathcal{F}$ at $\lambda > 0$, which according to Theorem \eqref{thm:diff-cF} is the linear operator $L_\lambda := D\mathcal{F}(\lambda): C^{2,\alpha}(\R)\to C^{0,\alpha}(\R)$ given by
\begin{align}
\label{eq:reg-func-cyliner}
\begin{split}
L_\lambda(w)(t)& =-w''(t)- \frac{1}{\lambda^2 } w(t)\\
&\quad -\gamma  \lambda \int_{\R}\int_{\mathbb{S}^1 } (w(t)-w(t-r)) G_\kappa \left((r^2+\lambda^2|e_1-\theta|^2)^{1/2} \right)\mathrm{d}r  \mathrm{d}\theta \\
&\quad +\frac{\gamma \lambda }{2} w(t) \int_{\R}\int_{\mathbb{S}^1}|e_1-\theta|^2 G_\kappa\left((r^2+\lambda^2|e_1-\theta|^2)^{1/2} \right) \mathrm{d}r  \mathrm{d}\theta .
\end{split}
\end{align}

\subsection{Recap on modified Bessel functions of the second kind} \label{sec:bessel-funct}
The spectral analysis of the operator $L_\lambda$ relies heavily on modified Bessel functions of the second kind.  We gather some of their essential properties from
\cite{BaEr74,GrRy07,Leb72} for the reader convenience.

\begin{definition}
In general, a function $w$ is said to be a modified Bessel function of order $\nu\in \mathbb{C}$ if it satisfies the differential equation
\begin{align*}
z^2\frac{d^2 w}{\mathrm{d}z^2} +z\frac{d w}{\mathrm{d}z} -(z^2+\nu^2)w=0.
\end{align*}
The modified Bessel functions of first kind $I_\nu$ and of second kind $K_\nu$ of order $\nu$ are
\begin{align*}
I_\nu(z) = \sum_{m=0}^{\infty} \frac{1}{m!\, \Gamma(m+\nu+1)}\big(\frac{z}{2}\big)^{2m+\nu}
\quad \text{and}\quad
K_\nu(z) =  \frac{\pi}{2}\frac{I_{-\nu}(z)-I_\nu(z)}{\sin(\nu\pi)},
\end{align*}
where  $K_n(z) = \lim\limits_{\nu \to n} K_\nu(z)$  for $n\in\mathbb{N}$ and  $\Gamma(x) =  \int_{0}^{\infty} t^{x-1} e^{-t}\mathrm{d}t$ is the Euler gamma function.
\end{definition}
\noindent According to \cite[p. 927, WA 95(20)]{GrRy07},
$I_\nu$ and $K_\nu$ are  linked by the relationship
\begin{align*}
I_\nu (z) K_{\nu+1} (z) + I_{\nu+1}(z) K_{\nu} (z) = \frac{1}{z}.
\end{align*}
\noindent For $\nu\in \mathbb{R}$, from \cite[p. 929, WA 93(3)]{GrRy07}, see also \cite{BaEr74}  we have
\begin{align}
&  r K'_\nu(r) +\nu K_\nu(r) = -r K_{\nu-1}(r)\quad\text{and} \quad K_{-\nu} = K_\nu.\label{eq:Bessel1}
\end{align}
In particular, taking $\nu = 0, 1$ relation \eqref{eq:Bessel1} yields,
\begin{align}
\label{eq:deriv-K10}
K'_0(r) = -K_1(r) \quad\text{and}\quad (rK_1)'(r) = -rK_0(r).
\end{align}
We are interested in the features of $K_0$ and $K_1$ which also have the following integral representation formulas.  Let $\gamma>0$, $\beta>0$ and $b>0$. It is shown in
\cite[p. 491, ET I 75(27), ET I 75(36)]{GrRy07} that the following relations holds
\begin{align}
\int_0^{\infty} \frac{e^{-\beta\sqrt{x^2+\gamma^2}}}{\sqrt{x^2+\gamma^2}}\sin(bx)\mathrm{d}x &=\frac{\gamma b}{\sqrt{\beta^2+b^2}} K_1\big(\gamma\sqrt{\beta^2+b^2}\big),  \label{eq:Bessel3}\\
\int_0^{\infty} \frac{e^{-\beta\sqrt{x^2+\gamma^2}}}{\sqrt{x^2+\gamma^2}}\cos(bx)\mathrm{d}x& =  K_0\big(\gamma\sqrt{\beta^2+b^2}\big).\label{eq:Bessel4}
\end{align}
In addition, by  \cite[p. 660, 12*]{GrRy07}  and \cite[p. 665, 5* \& 10*]{GrRy07}, for $b>0$ we have
\begin{align}
\int_0^{\infty} K_0(bx)\mathrm{d}x =  \frac{\pi}{2 b}, \quad \int_0^{\infty} xK_1(bx)\mathrm{d}x =  \frac{\pi}{2 b^2}\quad\text{and}\quad \int_0^{\infty} x^2 K_0(bx)\mathrm{d}x =  \frac{\pi}{2 b^3}.\label{eq:Bessel5}
\end{align}
From  \cite[p. 920,WA 231, 245(9)]{GrRy07} we find that
for all  $\nu \geq 0$,
\begin{align*}
K_\nu(r) = \sqrt{\frac{\pi}{2r}} e^{-r} \Big( 1 + \frac{4\nu^2-1}{8r} + \frac{(4\nu^2-1)(4\nu^2-9)}{128r^2} + O(r^{-3}) \Big)
\quad \text{as $r \to \infty$}.
\end{align*}
\noindent From  \cite[p. 919 ,WA 95(14)]{GrRy07} and  \cite[p. 919 ,WA 95(15)]{GrRy07}  it follows that
\begin{align} \label{eq:Bessel-asympK0}
K_0(r) &= -\log\left(\frac{r}{2}\right) - \gamma_e - \frac{r^2}{4}\log\left(\frac{r}{2}\right)
+ \frac{r^2}{4}(1 - \gamma_e) + O(r^4 \log(r)) \quad \text{as $r\to 0$}, \\
\label{eq:Bessel-asympK1}
K_1(r) &= \frac{1}{r} + \frac{r}{2}  \log\left(\frac{r}{2}\right) + \frac{r}{2} (\gamma_e- \frac{1}{2})+ O(r^3 \log(r)) \quad\quad  \text{as $r\to 0$}.
\end{align}
where $\gamma_e= -\Gamma'(1)$ is the Euler-Mascheroni constant. In particular, we have
\begin{align}
&\label{eq:decKnuInf0}
K_{0}(r)\sim -\log (r) \quad \text{and}\quad K_{1}(r)\sim   \frac{1}{r} \qquad\textrm{as \,\, $r\to 0$},
\\
\label{eq:decKnuInf}
&K_\nu(r) \sim \sqrt{\frac{\pi}{2r}} e^{-r} \qquad \text{as }\quad r\to \infty.
\end{align}
The  relations  \eqref{eq:Bessel4}, \eqref{eq:deriv-K10}, \eqref{eq:decKnuInf0} and \eqref{eq:decKnuInf} imply that  $K_i(r)>0$, $i=0,1$ for $r>0$.
We will need the following function.
\begin{definition}
Let $e_1 =(1,0)$  and  $\mathbb{S}^1= \partial B_1$ be the unit sphere of $\mathbb{R}^2$. For  $i=0,1$ and $n\geq i$  we define the function $\beta\mapsto T^i_n(\beta)$, $\beta>0$ by
\begin{align}
\label{eq:def-Tn1-Tn0}
T_n^i(\beta)=\int_{\mathbb{S}^1}|\theta-e_1|^n K_i(\beta|\theta-e_1|)\mathrm{d}\theta.
\end{align}
Note that, for $\theta =  (\theta_1, \theta_2) \in \mathbb{S}^1$, since  $|\theta - e_1|^2 = 2(1-\theta\cdot e_1)= 2(1-\theta_1)$, by a change of variable over the sphere $\mathbb{S}^1$, see for instance \cite[Appendix D.3]{Gra14}, together with the change of variavles $r= \beta\sqrt{2(1-\theta_1)}$ so that
$ \mathrm{d} \theta_1 = -\beta^{-2}rdr$, we have
\begin{align}
\begin{split}
T^i_n(\beta) &= 2 \int_{-1}^{1}(2(1-\theta_1))^{n/2}  (1-\theta_1^2)^{-1/2} K_i(\beta \sqrt{2(1-\theta_1)}) \mathrm{d} \theta_1\\
&=  4 \beta^{{-n-1}}\int_{0}^{2\beta } \big(4-\frac{r^2}{\beta^2}\big)^{-1/2} r^n K_i(r) \mathrm{d}r.  \label{eq:Expres-Tn}
\end{split}
\end{align}
Furthermore, using the formulas in \eqref{eq:Bessel1} we easily find that
\begin{align}\label{eq:derivat-Tin}
(-T^0_n)'(\beta)=  T^1_{n+1}(\beta)
\quad\text{ and }\quad
(-T^1_n)'(\beta)=  T^0_{n+1}(\beta)+\beta^{-1}T^1_{n}(\beta).
\end{align}
\end{definition}
\noindent We now provide asymptotic properties of $ T^i_n$.
\begin{lemma}\label{lem:sharp-est-sph-int-Bessel}
For $i\in \{0,1\}$ we let $n\geq i$ and define
\begin{align*}
B_n = \frac{\sqrt{\pi}\,\Gamma\left(\frac{n}{2}\right)}{2\,\Gamma\left(\frac{n+1}{2}\right)} \qquad \text{and} \qquad C^i_n = \int_{0}^\infty r^{n} K_i(r) \mathrm{d}r.
\end{align*}
Then, the following asymptotic limits hold
\begin{align*}
&\lim_{\beta\to \infty} \beta^2(\beta^{n+1}T^i_n(\beta)- 2C^i_n)= 2^{-2}C^i_{n+2},\\
&\lim_{\beta\to 0} (T^1_n(\beta)-2^{n+1}B_n \beta^{-1})=0,\\
&\lim_{\beta\to 0} (T^0_n(\beta)+ 2^{n+2}B_{n+1} (\log(\beta)+\gamma_e))=-2^{n+2}\frac{d}{d n} B_{n+1}.
\end{align*}
In particular we find that
\begin{align}
\label{eq:est-T10-b-infty}
T_n^i(\beta) &\sim  2\beta^{-n-1}C_n^i  +2^{-2}\beta^{-n-3}C_{n+2}^i
&&\textrm{ as $\beta\to \infty$},\\
\label{eq:est-T1-b0}
T^1_n(\beta) &\sim 2^{n+1}B_n\beta^{-1}\quad \text{and}\quad T^0_n(\beta) \sim
-2^{n+2} B_{n+1}\log(\beta)
&&\textrm{ as $\beta\to 0$}.
\end{align}
\end{lemma}
\begin{proof}
According to \eqref{eq:decKnuInf} we have
$ r^n K_i(r) \sim \sqrt{\frac{\pi}{2}} r^{n-\frac{1}{2}} e^{-r}$ as  $r\to \infty$, this implies that  $r\mapsto r^n K_i(r)$ for  $n\geq i$, is integrable on $[1,\infty).$ Now, if $i =0 $ then, $ n\geq 0$ so, the relation \eqref{eq:decKnuInf0} yields $r^nK_0(r) \sim -r^n\log(r)\in L^1(0,1)$ and hence $ r^nK_0(r)\in L^1(0,1)$ for $n\geq 0$. If  $i=1$ we have $n\geq 1$. The relation \eqref{eq:decKnuInf0} implies  $r^nK_1(r) \sim r^{n-1}$, we deduce that $r^n K_1(r)\in L^1(0, 1)$ for $n\geq 1$.   Thus we have shown that $C_n^i= \int_0^{\infty} r^nK_i(r)\mathrm{d}r$ converges for $n\geq i$. Next for the relation \eqref{eq:Expres-Tn} we can write
\begin{align}
\label{eq:decomp-Tin}
T^i_n(\beta) &= 4 \beta^{{-n-1}} T^i_{n,1}(\beta)+ 4\beta^{{-n-1}} T^i_{n,2}(\beta),
\end{align}
where define the terms $T^i_{n,1}(\beta)$ and $ T^i_{n,2}(\beta)$ by
\begin{align*}
T^i_{n,1}(\beta)
&= \int_{0}^{\sqrt{2}\beta} \big(4-\frac{r^2}{\beta^2}\big)^{-1/2} r^nK_i(r)
\mathrm{d}r,
\\
T^i_{n,2}(\beta)
&=\int_{\sqrt{2}\beta}^{2\beta} \big(4-\frac{r^2}{\beta^2}\big)^{-1/2}r^nK_i(r)
\mathrm{d}r.
\end{align*}
First of all, observe that $r\approx\beta$ for $r\in[\sqrt{2}\beta,2\beta]$ and $\beta $ large enough, the equivalence  \eqref{eq:decKnuInf} implies that $K_i(r)\leq c r^{-\frac12} e^{-r}\leq c \beta^{-\frac12} e^{-\sqrt{2}\beta }$  for some generic constant $c>0$. Hence we have
\begin{align*}
T^i_{n,2}(\beta) = \int_{\sqrt{2}\beta}^{2\beta}r^n \big(4-\frac{r^2}{\beta^2}\big)^{-1/2}K_i(r)\mathrm{d}r& \leq c   \beta^{{n-3/2}} \exp(-\sqrt{2}\beta) \int_{2\beta}^{\sqrt{2}\beta}  r\big(4-\frac{r^2}{\beta^2}\big)^{-1/2}\mathrm{d}r\\
&= c\sqrt{2}\beta^{{n+1/2}}\exp(-\sqrt{2} \beta).
\end{align*}
This estimate readily implies that
\begin{align}\label{eq:Expres-Tn2q}
\lim_{\beta\to \infty} \beta^q T^i_{n,2}(\beta)= 0\quad\qquad \textrm{for all $q\in \R$}.
\end{align}
In order to evaluate $T^i_{n,1}(\infty) = \lim\limits_{\beta\to\infty} T^i_{n,1}(\beta)$, we note that the function
\begin{align*}
f(r,\beta ) = \big(4-\frac{r^2}{\beta^2}\big)^{-1/2}r^n K_i(r)\mathds{1}_{[0,\sqrt{2}\beta]}(r),
\end{align*}
satisfies $f( r,\beta) \to
\frac{1}{2} r^n K_i(r)$ as $\beta\to \infty$ for $r>0$ and, since $\frac{1}{2}<\big(4-\frac{r^2}{\beta^2}\big)^{-1/2}<1$ for
$r\in [0,\sqrt{2}\beta]$, we have   $|f(r,\beta) |\leq  r^nK_i(r)\in L^1(0, \infty)$.  Hence we have   $r\mapsto r^nK_i(r)\in L^1(0, \infty)$ and the convergence dominated theorem implies
\begin{align*}
T^i_{n,1}(\infty)
&=\lim_{\beta\to \infty} \int_{0}^{\infty} \hspace*{-1ex}   \frac{r^n K_i(r)}{\big(4-\frac{r^2}{\beta^2}\big)^{1/2}} \mathds{1}_{[0,\sqrt{2}\beta]}(r)\mathrm{d}r=\frac12 \int_{0}^{\infty} \hspace*{-1ex} r^n K_i(r)\mathrm{d}r= \frac{C_n^i}{2}.
\end{align*}
Using this and fundamental theorem of calculus gives
\begin{align*}
T^i_{n,1}(\beta)&-T^i_{n,1}(\infty)
%
%
= \frac{1}{2\beta^2}\int_{0}^{\sqrt{2}\beta } \int_0^1 \big(4-\frac{r^2}{\beta^2}\varrho\big) ^{-3/2} r^{n+2} K_i(r)\mathrm{d}\varrho \mathrm{d}r
-\frac{1}{2}\int^{\infty}_{\sqrt{2}\beta } r^{n} K_i(r) \mathrm{d}r.
\end{align*}
Since $2 < 4 - (r^2/\beta^2)\varrho < 4$ for all $0 < r < \sqrt{2}\beta$ and $\varrho \in [0,1]$, applying, once again the convergence dominated theorem  as for $T^i_{n,1}(\infty)$, we obtain
\begin{align*}
\lim_{\beta\to \infty} \beta^2(T^i_{n,1}(\beta)-T^i_{n,1}(\infty))
&= 2^{-4}  \int^{\infty}_{0 }  r^{n+2} K_i(r)\mathrm{d}r = 2^{-4} C_{n+2}^i.
\end{align*}
Consequently, taking into account   \eqref{eq:Expres-Tn2q}  and the fact that $4T^i_{n,1}(\infty)= 2C^i_n$, we get
\begin{align*}
\lim_{\beta\to \infty} \beta^2(\beta^{n+1}T^i_n(\beta)- 2C^i_n)= \lim_{\beta\to \infty} \Big[4\beta^2(T^i_{n,1}(\beta)-T^i_{n,1}(\infty)) +4\beta^2T^i_{n,2}(\beta)\Big]= 2^{-2}C^i_{n+2}.
\end{align*}
From the foregoing, we deduce that
\begin{align*}
T^i_{n}(\beta)\sim 4\beta^{-n-1}T^i_{n,1}(\beta) \sim  2\beta^{-n-1}C_n^i + 2^{-2}\beta^{-n-3}C_{n+2}^i \quad\textrm{ as $\beta\to \infty$}.
\end{align*}
Now, we want to show \eqref{eq:est-T1-b0}. The change of variable implies  $r=2z\beta$ gives
\begin{align*}
T^i_n(\beta) &=  4 \beta^{-n-1} \int_{0}^{2\beta }\hspace{-1ex} \big(4-\frac{r^2}{\beta^2}\big)^{-1/2} r^n K_i(r) \mathrm{d}r =2^{n+2}\int_0^1 (1-z^2)^{-1/2} z^nK_i(2\beta z)\,\mathrm{d}z.
\end{align*}
By the duplication formula for the Gamma and Beta functions we  obtain
\begin{align*}
B_n=\int_0^1 \frac{z^{n-1}\mathrm{d}z }{ (1-z^2)^{1/2}}= \frac{1}{2} \int_0^1 t^{\frac{n}{2} - 1} (1-t)^{-\frac{1}{2}} \, \mathrm{d}t=\frac{1}{2} B\Big(\frac{n}{2}, \frac{1}{2}\Big) =\frac{\sqrt{\pi}\Gamma\left(\frac{n}{2}\right)}{2  \Gamma\left(\frac{n+1}{2}\right)}.
\end{align*}
The relations \eqref{eq:Bessel-asympK0} and
\eqref{eq:Bessel-asympK1}  together with the dominated convergence theorem imply
\begin{align*}
\lim_{\beta\to 0} (T^1_n(\beta)-2^{n+1}B_n \beta^{-1})=\lim_{\beta\to 0} \int_0^1 (1-z^2)^{-1/2} z^n\big[K_1(2\beta z) - \frac{1}{2\beta z} \big] \mathrm{d}z=0\quad n\geq 1,\\
\lim_{\beta\to 0}\int_0^1 (1-z^2)^{-1/2} z^n\big[K_0(2\beta z) +\log (\beta z)+ \gamma_e\big] \mathrm{d}z= 0\quad n\geq 0.
\end{align*}
In turn, since $\log(\beta z)= \log(\beta )+\log(z)$, the second limit equivalently shows\footnotemark that
\footnotetext{For $a>0$, we have $
\int_0^1 z^n \log(az) (1-z^2)^{-1/2} \mathrm{d}z = \tfrac{\sqrt{\pi} \, \Gamma\left(\frac{n+1}{2}\right)}{2  \Gamma\left(\frac{n+2}{2}\right)}
\left[ \log(a) + \tfrac{1}{2} \left( \psi\big(\tfrac{n+1}{2}\big) - \psi\big(\tfrac{n+2}{2}\big) \right) \right]
$ where $\psi=\frac{\Gamma'}{\Gamma}$ is the Digamma function.}
\begin{align*}
\lim_{\beta\to 0} (T^0_n(\beta)+ 2^{n+2}B_{n+1} (\log(\beta)+&\gamma_e))= -2^{n+2}\int_0^1\frac{z^n\log (z)\mathrm{d}z }{(1-z^2)^{1/2} }
= -2^{n+2}\frac{d}{d n} B_{n+1}.
\end{align*}
In particular, we readily  find that $T^1_n(\beta) \sim 2^{n+1}B_n\beta^{-1}$ and $T^0_n(\beta) \sim
-2^{n+2} B_{n+1}\log(\beta)$.
\end{proof}

\subsection{Eigenvalues and eigenfunctions of \texorpdfstring{$L_\lambda$}{L-lambda}}\label{subsec:spectral-decomp}
 The next lemma aims to investigate the spectrum of the linear operator  $L_\lambda: X\to Y$ given in \eqref{eq:reg-func-cyliner}.
\begin{lemma}\label{lem:def-linearized-eigen}
Let $\lambda > 0$. The functions  $e_k(t) := \cos(kt)$ and $\widetilde{e}_k(t):= \sin(kt)$ for $k \in \mathbb{N}$ are eigenfunctions of $L_\lambda$ To be more precise we have  \begin{align*}
L_\lambda(e_k) = \sigma_{\lambda, \gamma}(k) e_k\quad\text{and}\quad L_\lambda (\widetilde{e}_k)
(t)= \sigma_{\lambda,\gamma}(k) \, \widetilde{e}_k(t)
\end{align*}
where the eigenvalue $\sigma_{\lambda ,\gamma}(k)$ is given by
\begin{align*}
\sigma_{\lambda ,\gamma}(k)
=  & k^2- \frac{1}{\lambda ^2}\\
 & - 2 \gamma \lambda \int_{\mathbb{S}^1}   K_0( \lambda  |e_1-\theta |\kappa )
\mathrm{d} \theta + 2\gamma \lambda\int_{\mathbb{S}^1}
K_0( \lambda  |e_1-\theta |\sqrt{  \kappa^2+  k^2 } )\mathrm{d} \theta
\\
&+\frac{\gamma \lambda }{2}
\int_{\mathbb{S}^1}  |e_1-\theta |^2 K_0( \lambda  |e_1-\theta | \kappa ) \mathrm{d} \theta.
\end{align*}
In short,  with the notation \eqref{eq:derivat-Tin} we have
\begin{align}
\label{eq:new-eigen-val}
\sigma_{\lambda ,\gamma}(k)=k^2- \frac{1}{\lambda ^2}-2\gamma \lambda T^0_0(\lambda  \kappa)+2\gamma\lambda T^0_0(\lambda \sqrt{\kappa^2+k^2}) +\gamma\lambda  T^0_2(\lambda  \kappa).
\end{align}
\end{lemma}
\begin{proof}
Note that $-e_k''- \frac{1}{\lambda ^2} e_k = \big( k^2 -\frac{1}{\lambda^2}\big)e_k$. Thus, by the expression  \eqref{eq:reg-func-cyliner} we find that,
\begin{align*}
L_\lambda (e_k)(t)
& =\big(k^2 -\frac{1}{\lambda^2}\big)e_k(t)\\
&\quad - \gamma  \lambda   \int_{\mathbb{R}}\int_{\mathbb{S}^1}(e_k(t)-e_k(t-r)) G_\kappa\left((r^2+\lambda ^2|e_1-\theta|^2)^{1/2} \right)\mathrm{d}r \mathrm{d}\theta\\
&\quad +\frac{\gamma \lambda }{2}  e_k(t) \int_{\mathbb{R}}\int_{\mathbb{S}^1}|e_1-\theta|^2 G_\kappa\left((r^2+\lambda ^2|e_1-\theta|^2)^{1/2} \right) \mathrm{d}r \mathrm{d}\theta   .
\end{align*}
The identity $\cos(a-b)=\cos(a)\cos(b) + \sin(a)\sin(b)$ implies,
\begin{align}
e_k(t)-e_k(t-r)
&=e_k(t)\int_0^1 kr\sin(\varrho kr)\mathrm{d}\varrho - \sin(kt)\sin(kr).\label{eq:Eigen-sec-1}
\end{align}
Since, $r\mapsto\sin(kr) G_\kappa\left((r^2+\lambda ^2|e_1-\theta|^2)^{1/2} \right)$ is odd, its integral over $\mathbb{R}$ vanishes. Therefore, the  expression of $L_\lambda(e_k)(t)$ amounts to $L_\lambda(e_k) = \sigma_{\lambda, \gamma}(k) e_k$ where,
\begin{align}
\sigma_{\lambda ,\gamma}(k)
& = k^2- \frac{1}{\lambda ^2} \nonumber
\\
& \quad - \gamma  \lambda  k \int_0^1 \int_{\mathbb{R}}\int_{\mathbb{S}^1}r\sin(\varrho k r) G_\kappa\big((r^2+\lambda ^2|e_1-\theta|^2)^{1/2} \big)\mathrm{d}r \mathrm{d}\theta  \mathrm{d}\varrho\label{eq:I1}\\
&  \quad +\frac{\gamma \lambda  }{2}  \int_{\mathbb{R}}\int_{\mathbb{S}^1}|e_1-\theta|^2 G_\kappa\left((r^2+\lambda ^2|e_1-\theta|^2)^{1/2} \right) \mathrm{d}r \mathrm{d}\theta. \label{eq:I2}
\end{align}
\noindent Now let $p_\theta^2:=|e_1-\theta|^2 = 2(1-\theta.e_1)$ and recall that $G_\kappa(r)= r^{-1}e^{-\kappa r}$.  For the term in  \eqref{eq:I2},  the convergence dominated theorem together with \eqref{eq:Bessel4} and  the continuity of $K_0$ imply
\begin{align} \label{eq:Bessel4-bis}
\int_0^{\infty} \frac{e^{-\beta\sqrt{x^2+\gamma^2}}}{\sqrt{x^2+\gamma^2}}\mathrm{d}x
&=\lim_{b\to 0}\int_0^{\infty} \frac{e^{-\beta\sqrt{x^2+\gamma^2}}}{\sqrt{x^2+\gamma^2}}
\cos(bx)\mathrm{d} x = \lim_{b\to 0} K_0\big(\gamma\sqrt{\beta^2+b^2}\big) = K_0(\gamma\beta),
\end{align}
for all for $\gamma>0$ and $\beta>0$. Applying this, yields that
\begin{align}\label{eq:Bessel-Eiegn9}
\int_{\mathbb{R}}    G_\kappa\left((r^2+(\lambda  p_\theta)^2)^{1/2} \right)\mathrm{d}r =2 \int_{0}^{\infty} \frac{e^{-\kappa\sqrt{r^2+(\lambda p_\theta)^2}}}{\sqrt{r^2+(\lambda p_\theta)^2}}\mathrm{d}r = 2K_0( \lambda  \kappa p_\theta ).
\end{align}
On the other hand, from the relations \eqref{eq:deriv-K10} and  \eqref{eq:Bessel3} we deduce
\begin{align*}
k\int_{\mathbb{R}} &r\sin(\varrho k r)  G_\kappa\left((r^2+(\lambda  p_\theta)^2)^{1/2} \right)\mathrm{d}r
= 2 k\int_0^{\infty} \frac{re^{-\kappa\sqrt{r^2+(\lambda p_\theta)^2}}}{\sqrt{r^2+(\lambda p_\theta)^2}}
\sin(\varrho k r)\mathrm{d}r\\
&=
2 \frac{k^2 \varrho \lambda  p_\theta}{\sqrt{  \kappa^2+\varrho^2 k^2 }} K_1( \lambda  p_\theta\sqrt{  \kappa^2+\varrho^2 k^2 } )
\overset{K_1 =-K'_0}{=} -2\frac{d}{\mathrm{d}\varrho}\big(K_0(\lambda  p_\theta
\sqrt{\kappa^2+\varrho^2 k^2})\big).
\end{align*}
It follows that, the integrand factor under $\mathbb{S}^1=\partial B_1$ in \eqref{eq:I1} is given by
\begin{align}\label{eq:Bessel-Eiegn 8}
k \int_0^1\int_{\mathbb{R}} r\sin(\varrho k r)  G_\kappa((r^2+(\lambda  p_\theta)^2)^{1/2} )\mathrm{d}r\mathrm{d}\varrho = 2K_0(\lambda p_\theta\kappa)-2K_0(\lambda p_\theta\sqrt{\kappa^2+k^2}).
\end{align}
Substituting \eqref{eq:Bessel-Eiegn 8} and \eqref{eq:Bessel-Eiegn9} into \eqref{eq:I1} and \eqref{eq:I2}, respectively, gives the desired expression for $\sigma_{\lambda ,\gamma}(k)$. By the token one obtains $L_\lambda (\widetilde{e}_k) (t)=\sigma_{\lambda ,\gamma}(k) \widetilde{e}_k(t).$
\end{proof}
\noindent In the sequel, we provide qualitative properties for the eigenvalues $\sigma_{\lambda,\gamma}(k)$ for $k\in \mathbb{N}$.
%
\begin{proposition}\label{prop:asym-eigen-val}
Let $\gamma >0$, $\lambda>0$ and $k\in \mathbb{N}$. The following hold
\begin{align}
\lim_{k \to \infty} \frac{\sigma_{\lambda ,\gamma}(k)}{k^2}&= 1.\label{eq:eigen-p3}
\\
\sigma_{0 ,\gamma}(k)=  \lim_{\lambda\to 0} \sigma_{\lambda ,\gamma}(k)&= -\infty.\label{eq:eigen-p1}\\
\sigma_{\infty ,\gamma}(k) = \lim_{\lambda\to \infty} \sigma_{\lambda ,\gamma}(k)&=k^2 -2\pi\gamma \Big(  \frac{1}{\kappa}-\frac{1}{\sqrt{\kappa^2+k^2}}\Big).
\label{eq:eigen-p2}
\end{align}
\end{proposition}

\begin{proof}
Since  $C^0_0 = \int_{0}^{\infty} K_0(r)\mathrm{d}r = \frac{\pi}{2}$; see the formula \eqref{eq:Bessel5},  Lemma~\ref{lem:sharp-est-sph-int-Bessel} implies that,
\begin{align*}
&\lim_{\lambda\to \infty} 2\gamma \lambda T^0_0(\lambda a)
=  \lim_{\lambda\to \infty} 2\gamma \big(2(\lambda a)^{-1} C^0_0 +2^{-2}(\lambda a)^{-3}C^0_2\big)
= \frac{2\pi \gamma}{a}, \\
&\lim_{\lambda\to \infty} 2\gamma \lambda T^0_2(\lambda a)
=  \lim_{\lambda\to \infty} 2\gamma \big(2(\lambda a)^{-3} C^0_2 +2^{-2}(\lambda a)^{-5}C^0_4\big)=0,
\end{align*}
for any $a>0$.  It follows that
\begin{align*}
\lim_{\lambda\to \infty}
\sigma_{\lambda ,\gamma}(k)&=\lim_{\lambda\to \infty} \Big(k^2- \frac{1}{\lambda ^2}-2\gamma \lambda T^0_0(\lambda  \kappa)+2\gamma\lambda T^0_0(\lambda \sqrt{\kappa^2+k^2}) +\gamma\lambda  T^0_2(\lambda  \kappa)\Big)\\
&= k^2 -  2\pi \gamma\Big(  \frac{1}{\kappa}-\frac{1}{\sqrt{\kappa^2+k^2}}\Big).
\end{align*}
Analogously, Lemma \ref{lem:sharp-est-sph-int-Bessel} implies
\eqref{eq:eigen-p1} as follows
\begin{align*}
\lim_{\lambda\to 0}
\sigma_{\lambda ,\gamma}(k)
&=\lim_{\lambda\to0} \big(k^2- \frac{1}{\lambda ^2}
+O(\lambda \log(\lambda \sqrt{\kappa^2+k^2})) +
O(\lambda  \log(\lambda  \kappa))\big)=-\infty,\\
\lim_{k\to \infty}
\frac{\sigma_{\lambda ,\gamma}(k)}{k^2}
&=1+ \lim_{k\to \infty} \big(- \frac{1}{k^2\lambda ^2}
+O(k^{-2}) +O(k^{-2}  ( \sqrt{\kappa^2+ k^2})^{-1})\big)=1.
\end{align*}
\end{proof}

\begin{lemma}\label{lem:monoton-lambda}
There exists $\gamma_0\in (0, \infty)$ such that
\begin{align*}
\partial_\lambda    \sigma_{\lambda ,\gamma}(1)>0 \qquad \textrm{ for every $\lambda >0$ and $\gamma\in (0,\gamma_0)$.}
\end{align*}
\end{lemma}

\begin{proof}
From \eqref{eq:new-eigen-val} we know that,
\begin{align*}
\sigma_{\lambda, \gamma }(1)&=1-\frac{1}{\lambda^2}-\gamma\big( 2\lambda T^0_0(\lambda  \kappa)-2\lambda T^0_0(\lambda \sqrt{\kappa^2+1}) -\lambda  T^0_2(\lambda  \kappa) \big).
\end{align*}
By direct computation, using the relations \eqref{eq:derivat-Tin}, we obtain
\begin{align*}
\partial_{\lambda }\sigma_{\lambda ,\gamma}(1) = \frac{2}{\lambda^3} &- \gamma \Big( 2 T^0_0(\lambda \kappa) - 2 T^0_0(\lambda \sqrt{\kappa^2+1}) - T^0_2(\lambda \kappa) \\
&\qquad- 2\lambda \kappa T^1_1(\lambda\kappa) + 2\lambda \sqrt{\kappa^2+1} T^1_1(\lambda \sqrt{\kappa^2+1}) + \lambda \kappa T^1_3(\lambda \kappa) \Big).
\end{align*}
We therefore have
\begin{align}\label{Deriv-Sig-la-cub}
\lambda  ^3 \partial_{\lambda }\sigma_{\lambda ,\gamma}(1)=2- \gamma A(\lambda ) ,
\end{align}
where  $A: (0,\infty)\to \mathbb{R}$ is defined by
\begin{align*}
A(\lambda )&:=
 2\lambda^3 T^0_0(\lambda  \kappa)-2\lambda^3 T^0_0(\lambda \sqrt{\kappa^2+1}) -\lambda ^3 T^0_2(\lambda  \kappa) \\
&\quad- 2\lambda^4 \kappa  T^1_1(\lambda  \kappa)+2\lambda^4\sqrt{\kappa^2+1} T^1_1(\lambda \sqrt{\kappa^2+1}) +\lambda ^4 \kappa T^1_3(\lambda  \kappa).
\end{align*}
We have $\lim_{\lambda \to 0}A(\lambda )=0$ since using the  relation \eqref{eq:est-T1-b0} of Lemma \ref{lem:sharp-est-sph-int-Bessel}, as $\lambda\to 0$ implies
\begin{align*}
A(\lambda)& \approx 2^4\lambda^3( -\log(\lambda  \kappa) +\log(\lambda \sqrt{\kappa^2+1}) +\frac{2}{3} \log(\lambda  \kappa)   ) \\
&+2^3\lambda^4( -(\lambda  \kappa)^{-1} +(\lambda \sqrt{\kappa^2+1})^{-1} +\frac{2}{3}\kappa (\lambda  \kappa)^{-1} ) \to 0\quad \text{as $\lambda\to 0$}.
\end{align*}
Now rewriting the expression of $A(\lambda)$ and using  the relation \eqref{eq:est-T10-b-infty} as $\lambda \to \infty$ yields
\begin{align*}
A(\lambda )
&=  2\lambda^3 T^0_0(\lambda  \kappa)-2\lambda^3 (\lambda\kappa)  T^1_1(\lambda  \kappa)\\
&\quad +   2\lambda^3 (\lambda\sqrt{\kappa^2+1}) T^1_1(\lambda \sqrt{\kappa^2+1})-2\lambda^3 T^0_0(\lambda \sqrt{\kappa^2+1})\\
&\quad +\lambda^3  (\lambda\kappa)T^1_3(\lambda  \kappa)-\lambda ^3 T^0_2(\lambda  \kappa)\\
&\approx 4\lambda ^3(\lambda \kappa)^{-1}\left(C_0^0-C_1^1 \right) +2^{-1}\lambda^3(\lambda\kappa)^{-3} (C_2^0-C_3^1) \\
&\quad +4  \lambda ^3(\lambda \sqrt{\kappa^2 + 1})^{-1}\left(C_1^1-C_0^0\right)+2^{-1}\lambda ^3(\lambda \sqrt{\kappa^2 + 1})^{-3}\left(C_3^1-C_2^0\right)\\
& \quad +2\lambda ^3(\lambda \kappa)^{-3}\left(C_3^1-C_2^0 \right) +2^{-2}\lambda^3(\lambda\kappa)^{-5} (C_5^1-C_4^0) ,
\end{align*}
where we recall that $C_n^i= \int_0^{\infty } r^nK_i(r) \mathrm{d}r$, $n\geq i$. From \eqref{eq:deriv-K10}, \eqref{eq:decKnuInf0} and \eqref{eq:decKnuInf} we have $K_0'(r) = -K_1(r)$, $K_0(r)\sim -\log(r)$ as $r\to 0$
and $K_0(r)\sim \sqrt{\frac{\pi}{2r}}e^{-r} $ as $r\to \infty$. Thus, the integration by part together with the formulas in relation \eqref{eq:Bessel5}  imply
\begin{align*}
C_0^0&=\int_0^{\infty} K_0(r)\mathrm{d}r = -\int_0^{\infty} rK'_0(r)\mathrm{d}r=  \int_0^{\infty} rK_1(r)\mathrm{d}r= C_1^1= \frac{\pi}{2}, \\
C^1_3&=\int_0^{\infty} r^3K_1(r)\mathrm{d}r= -\int_0^\infty r^3K'_0(r)\mathrm{d}r= 3 \int_0^{\infty} r^2K_0(r)\mathrm{d}r = \frac{3\pi}{2} = 3C_2^0.
\end{align*}
That is, $C_1^1=C_0^0$ and $C_3^1 = 3C_2^0$.
From the foregoing we  have
\begin{align*}
\lim_{\lambda \to \infty } A(\lambda)
&= -\kappa^{-3}C_2^0 + (\sqrt{\kappa^2 +1 })^{-3}C_2^0 + 4\kappa^{-3}C_2^0 = \frac{\pi}{2}\big(3\kappa^{-3} +(\sqrt{\kappa^2 +1 })^{-3} \big).
\end{align*}
Note that, the functions $T^i_n $ for $ i=0,1,~n\geq i$ are continuous, on $(0,\infty)$.  In conclusion, $\lambda \mapsto A(\lambda)$ continuous on $[0,\infty)$ since $A(0)=\lim_{\lambda \to 0}A(\lambda )=0$; and we have
\begin{align*}
A(\infty):= \lim_{\lambda \to \infty }A(\lambda )=   \frac{\pi}{2}\big(3\kappa^{-3} +(\sqrt{\kappa^2 +1 })^{-3} \big)>0.
\end{align*}
It follows that $A$ is bounded on $[0,\infty)$. Thus $\sup_{\lambda \in (0,\infty)} A(\lambda) $ exists and we have
\begin{align*}
A_0 := \sup_{\lambda \in (0,\infty)} A(\lambda) \geq A(\infty)> 0.
\end{align*}
Recalling  the relation \eqref{Deriv-Sig-la-cub}, it follows that for $\gamma <  \frac{2}{A_0}$ we have
 $$ {\lambda^3 }  \partial_\lambda   \sigma_{\lambda ,\gamma}(1)=2- \gamma  A(\lambda )\geq \inf_{\lambda \in (0,\infty)} 2-\gamma A(\lambda ) =  2-\gamma A_0>0.$$
 This is equivalent to stating that $\partial_\lambda \sigma_{\lambda ,\gamma}(1) > 0$ for all $\gamma < \frac{2}{A_0}$. Consequently, we set
\begin{align*}
\gamma_0:=\frac{2}{A_0}= \frac{2}{\sup_{\lambda \in (0,\infty)} A(\lambda)}>0.
\end{align*}
\end{proof}

\noindent The following theorem provides the key result for the eigenvalues $\sigma_{\lambda,\gamma}(k)$.
\begin{theorem}\label{thm:Lambda-star}
There exists $\gamma_* \in \big(0, \, \frac{1}{2\pi} \tfrac{1}{\frac{1}{\kappa} - \frac{1}{\sqrt{\kappa^2+1}}}\big)$
such that the following assertions hold.
\begin{itemize}
\item For every $\gamma\in (0,\gamma_*)$  and $\lambda>0$ the sequence $k\mapsto \sigma_{\lambda,\gamma}(k)$ is strictly increasing.
\item
For each $\gamma\in (0,\gamma_*)$ there exists a unique parameter $\lambda _*=\lambda _*(\gamma)>0$  satisfying
\begin{align}
\sigma_{\lambda _*(\gamma),\gamma}(1)=0.
\end{align}
In particular, it satisfies
\begin{align}\label{eq:sk-neq-0}
\sigma_{\lambda _*(\gamma),\gamma}(0)<0=\sigma_{\lambda _*(\gamma),\gamma}(1)<\sigma_{\lambda _*(\gamma),\gamma}(k)  \quad \textrm{ for all $k\geq 2$}
\end{align}
\end{itemize}
\end{theorem}
\begin{proof}
It follows  from  the relation \eqref{eq:eigen-p1}  that
\begin{align*}
\sigma_{\infty ,\gamma}(1) = 1- 2\pi\gamma\Big(  \frac{1}{\kappa}-\frac{1}{\sqrt{\kappa^2+1}}\Big)>0 \quad\textrm{if and only if }\quad \gamma< \frac{1}{2\pi} \frac{1}{ \frac{1}{\kappa}-\frac{1}{\sqrt{\kappa^2+1}}}.
\end{align*}
Let $\gamma_0\in (0, \infty)$ be  given by  Lemma \ref{lem:monoton-lambda} and let us fix $\gamma\in (0,\gamma_1)$ where we define
\begin{align}\label{eq:gamma1}
\gamma_1:= \min\Big(\gamma_0,  \frac{1}{2\pi} \frac{1}{ \frac{1}{\kappa}-\frac{1}{\sqrt{\kappa^2+1}}  } \Big).
\end{align}
However, since $\gamma<\gamma_1\leq \gamma_0$, by Lemma \ref{lem:monoton-lambda} we have
\begin{align*}
\partial_\lambda\sigma_{\lambda, \gamma}(1)>0 \qquad\textrm{for every $\lambda>0$}.
\end{align*}
In other words, the map $\lambda\mapsto \sigma_{\lambda, \gamma}(1)$ is strictly increasing over $(0,\infty)$
and $\sigma_{\infty, \gamma}(1)\cdot \sigma_{0, \gamma}(1)<0.$ since $\sigma_{\infty, \gamma}(1) >0 $, $\gamma\in (0,\gamma_1)$, while from \eqref{eq:eigen-p1} we have $\sigma_{0 ,\gamma}(1) = -\infty$. The intermediate value theorem implies that, there is a  unique parameter $\lambda_*=\lambda_*(\gamma)>0$ such that,
\begin{align*}
\sigma_{\lambda _*(\gamma),\gamma}(1)=0 .
\end{align*}
Next, we want to study the monotonicity of $k\mapsto \sigma_{\lambda ,\gamma}(k)$. Given that $(T^0_0)' = -T^1_1$ (see \eqref{eq:derivat-Tin}), the fundamental theorem of calculus implies
\begin{align*}
-2\gamma \lambda T^0_0(\lambda  \kappa)+2\gamma\lambda T^0_0(\lambda \sqrt{\kappa^2+k^2})
&=2\gamma\lambda\int_0^{1} \frac{d}{\mathrm{d}\varrho}\Big(  T^0_0(\lambda \sqrt{\kappa^2+k^2\varrho}) \Big)\mathrm{d}\varrho\\
%
%
&= -\gamma \int_0^1\frac{ \lambda^2 k^2}{\sqrt{  \kappa^2+k^2\varrho }}T^1_1(\lambda  \sqrt{  \kappa^2+k^2\varrho  })\mathrm{d}\varrho.
\end{align*}
Inserting this in the  relation \eqref{eq:new-eigen-val} we have,
\begin{align}\label{eq:new-exp-eigen}
\sigma_{\lambda,\gamma}(k)=k^2- \frac{1}{\lambda ^2}- \gamma \int_0^1\frac{ \lambda^2 k^2}{\sqrt{  \kappa^2+k^2\varrho  }}T^1_1\big(\lambda  \sqrt{  \kappa^2+k^2\varrho}\big) \mathrm{d}\varrho+\gamma\lambda  T^0_2(\lambda  \kappa).
\end{align}
Now, we introduce the functions
$f_{\lambda ,\gamma}, \, g_{\lambda ,\gamma}: [0,\infty)\to \mathbb{R}$ with
\begin{align*}
f_{\lambda ,\gamma}(x)&= 1-\gamma \int_0^1\frac{ \lambda^2}{\sqrt{  \kappa^2+x^2\varrho  }}T^1_1\big(\lambda  \sqrt{  \kappa^2+x^2\varrho}\big) \mathrm{d}\varrho,
\\
g_{\lambda ,\gamma}(x)&= x^2 f_{\lambda,\gamma}(x)-\frac{1}{\lambda ^2} +\gamma\lambda  T^0_2(\lambda  \kappa).
\end{align*}
It is clear that $\sigma_{\lambda ,\gamma}(k)=g_{\lambda ,\gamma}(k)$ for all $k \in \mathbb{N}$.
A straightforward computation gives
\begin{align*}
f'_{\lambda ,\gamma}(x) &= \int_0^1\hspace{-1ex} \frac{\gamma x\lambda^2 \varrho}{(\kappa^2 + x^2\varrho)^{3/2}} T_1^1(\lambda\sqrt{\kappa^2+x^2\varrho})
+ \frac{ \gamma x\lambda^3}{\kappa^2 + x^2\varrho} (-T_1^1)'(\lambda\sqrt{\kappa^2+x^2\varrho})\mathrm{d}\varrho.
\end{align*}
Since $T_n^i $ are non-negative functions because $K_i>0$, using
\eqref{eq:derivat-Tin} we get
\begin{align*}
(- T_1^1)'(\beta) = T_2^0(\beta) +\beta^{-1}T_1^1(\beta)>0   \qquad\textrm{for every $\beta>0$}.
\end{align*}
Therefore, for every $\gamma > 0$, the expression for $f'_{\lambda,\gamma}(x)$ shows that $f'_{\lambda ,\gamma}(x) \ge 0$ for all $x \ge 0$. Consequently, the mapping $x \mapsto f_{\lambda ,\gamma}(x)$ is non-decreasing on $[0,\infty)$.
Using in particular that $f'_{\lambda,\gamma}(x) \ge 0$ and consequently the fact that $f_{\lambda,\gamma}(x) \geq f_{\lambda,\gamma}(0)$ we obtain
\begin{align}\label{eq:deriv-g}
g'_{\lambda ,\gamma}(x)&=2x f_{\lambda,\gamma}(x)+x^2f'_{\lambda ,\gamma}(x)
\geq 2x f_{\lambda ,\gamma}(0)\quad \text{for all $x>0$},
\end{align}
where $f_{\lambda ,\gamma}(0)=1- \gamma    \frac{ \lambda ^2}{\kappa}    T_1^1(\lambda\kappa).$ However, Lemma \ref{lem:sharp-est-sph-int-Bessel} together with \eqref{eq:Bessel5} imply
\begin{align*}
&\lim_{\lambda\to 0} \frac{ \lambda ^2 }{\kappa } T_1^1(\lambda\kappa) = \lim_{\lambda\to 0} -\frac{ \lambda ^2 }{\kappa } \log(\lambda\kappa) = 0,\\
&\lim_{\lambda\to \infty} \frac{ \lambda ^2}{\kappa} T_1^1(\lambda\kappa) = \lim_{\lambda\to \infty} 2\frac{ \lambda ^2 }{\kappa} (\lambda\kappa)^{-2}\int_0^{\infty} rK_1(r)\mathrm{d}r   =  \pi\kappa^{-3}.
\end{align*}
Thus, the continuous $\lambda\mapsto \frac{ \lambda ^2 }{\kappa } T_1^1(\lambda\kappa) $ is also  bounded on $[0, \infty)$; enabling us to consider
\begin{align*}
A_1:= \kappa^{-1}\sup\limits_{\lambda>0}\left\lbrace\lambda ^2 T_1^1(\lambda\kappa)\right\rbrace\geq \lim_{\lambda\to \infty}\kappa^{-1}\lambda ^2 T_1^1(\lambda\kappa)= \pi\kappa^{-3}>0.
\end{align*}
So that for $0<\gamma <\frac{1}{A_1}$ and all for all $x>0$  we have
\begin{align*}
g'_{\lambda,\gamma}(x)\geq 2x f_{\lambda ,\gamma}(0)= 2x\big(1- \gamma    \frac{ \lambda ^2}{\kappa}   T_1^1(\lambda\kappa)\big)\geq 2x(1-\gamma A_1)>0.
\end{align*}
In other words,  for all $\lambda > 0$ and $0<\gamma < \frac{1}{A_1}$,  $x \mapsto g_{\lambda,\gamma}(x)$ is strictly increasing. Hence $k \mapsto \sigma_{\lambda,\gamma}(k) = g_{\lambda,\gamma}(k)$ is strictly increasing.
In conclusion, it is sufficient to define
\begin{align*}
\gamma_*:=\min (\gamma_1,\frac{1}{A_1})=  \min\Big(\gamma_0, \frac{1}{A_1}, \frac{1}{2\pi} \frac{1}{ \frac{1}{\kappa}-\frac{1}{\sqrt{\kappa^2+1}}  } \Big).
\end{align*}
Let us  fix $\gamma\in (0,\gamma_*)$ then, $\gamma< \gamma_1$ and  consider $\lambda_*=\lambda_*(\gamma)$ which is the unique so that  $g_{\lambda _*,\gamma}(1)=\sigma_{\lambda _*,\gamma}(1)=0.$ As the sequence  $k\mapsto g_{\lambda _*,\gamma}(k)=\sigma_{\lambda _*,\gamma}(k) $ is strictly  increasing and $g_{\lambda _*,\gamma}(1)=0$, we find that $g_{\lambda _*,\gamma}(0)<0$ and hence
\begin{align*}
\sigma_{\lambda _*,\gamma}(0)< 0= \sigma_{\lambda _*,\gamma}(1)< \sigma_{\lambda _*,\gamma}(k)\qquad\textrm{for every $k\geq 2$}.
\end{align*}
\end{proof}
\begin{corollary}
\label{cor:eigenfunct-of-Llam}
Let $\gamma_* > 0$ be given as in Theorem~\ref{thm:Lambda-star}, and fix $\lambda > 0$ and $\gamma \in (0, \gamma_*)$.
\begin{itemize}
\item All eigenfunctions of the operator $L_\lambda: X \to Y$ are of the form
\begin{align*}
\varphi_k(t) = c e_k(t), \quad c \in \R, \qquad \text{where } e_k(t) = \cos(kt) \text{ for } k \in \mathbb{N}.
\end{align*}
\item The sequence of eigenvalues $ \{\sigma_{\lambda,\gamma}(k)\}_{k \in \mathbb{N}}$ is strictly increasing with
$ \sigma_{\lambda,\gamma}(k)\to  \infty $ as $k\to\infty$. In particular, each  $\sigma_{\lambda,\gamma}(k)$  is a simple eigenvalue.
\item Moreover, the kernel of $L_{\lambda_{*}}$ is given by $N (L_{\lambda_{*}})= \R\cos(\cdot)$.
\end{itemize}
\end{corollary}
\begin{proof}
Let $\varphi\in X\setminus\{0\}$ be an eigenfunction of $L_\lambda $ with respect an eigenvalue $b_{\lambda,\gamma}$. The function $\varphi\in X$  is periodic, $C^2$ and even so, by Fourier decomposition it can written as
\begin{align}\label{eq:eigen-fourier-decom}
\varphi(t) =\sum_{k=0}^{\infty} c_k e_k(t) \quad
\quad c_k\in\R.
\end{align}
Necessarily,   there is $k_0\in\mathbb{N}$ such that, $c_{k_0}\neq 0$ since  $\varphi\neq 0$. Next, by Lemma \ref{lem:def-linearized-eigen}, $L_\lambda(e_k) = \sigma_{\lambda,\gamma}(k)e_k$. By the linearity and continuity of $L_\lambda$, we get $L_\lambda(\varphi) = b_{\lambda,\gamma}\varphi$ if and only if
\begin{align}
\sum_{k=0}^{\infty}c_k\left( b_{\lambda,\gamma}-\sigma_{\lambda,\gamma}(k)\right) e_k = 0.
\end{align}
The uniqueness of Fourier decomposition, leads to the system of equation
\begin{align}
c_k\left( b_{\lambda,\gamma}-\sigma_{\lambda,\gamma}(k)\right)= 0\qquad\textrm{ for every $k\in\mathbb{N}.$}
\end{align}
  Since $c_{k_0}\neq 0$ we get  $b_{\lambda,\gamma}=\sigma_{\lambda,\gamma}(k_0). $ Meanwhile, Proposition \eqref{prop:asym-eigen-val} implies $\sigma_{\lambda,\gamma}(k)\to \infty$
as $k\to \infty$, while from Theorem \ref{thm:Lambda-star} the map $k\mapsto \sigma_{\lambda,\gamma}(k)$ is strictly increasing. Hence,
\begin{align*}
b_{\lambda,\gamma}:=\sigma_{\lambda,\gamma}(k_0) \not=\sigma_{\lambda,\gamma}(k)\quad\textrm{ for every $k\not= k_0$}.
\end{align*}
It follows that  $c_k  = 0$ when $k\not=k_0$.
Inserting all in \eqref{eq:eigen-fourier-decom}  we deduce that $ \varphi(t) = c_{k_0}e_{k_0}(t)$  and $b_{\lambda,\gamma}
=\sigma_{\lambda,\gamma}(k_0)$. The latter also implies that each $\sigma_{\lambda, \gamma}(k)$ eigenvalue is simple.
In particular, since $ L_{\lambda_*}(\cos(\cdot)) = \sigma_{\lambda_*, \gamma}(1)\cos(\cdot) = 0$ by Theorem~\ref{thm:Lambda-star}, the simplicity of$\sigma_{\lambda_*, \gamma}(1)$ implies that  we have  $\varphi \in N(L_{\lambda_*})= \R\cos(\cdot)$.
\end{proof}

\section{Proof of Theorem \ref{thm:main-thm}: Existence of periodic interface}\label{sec:main-result}
Recall that the spaces $X$ and $Y$ are subspaces of even and periodic functions of H\"older's spaces $C^{2,\alpha}(\mathbb{R})$ and $C^{0,\alpha}(\mathbb{R}),$ respectively. We introduce the closed hyperspaces
\begin{align*}
X^\perp &:= \Big\{u\in X \,:\, \int_0^{2\pi}u(t)\cos(t) \mathrm{d}t = 0\Big\},
\\
Y^\perp &:= \Big\{u\in Y \,:\, \int_0^{2\pi}u(t)\cos(t)\mathrm{d}t = 0\Big\}.
\end{align*}
Note that $X^\perp$ and $Y^\perp$ represent the kernels of the continuous linear form
$u \mapsto \int_0^{2\pi} u(t)\cos(t)\mathrm{d}t$
on $X$ and $Y$, respectively.
Moreover, $\cos(\cdot)$ belongs to neither $X^\perp$ nor $Y^\perp$. In other words, $X^\perp$ and $Y^\perp$ are hyperplanes in $X$ and $Y$, with common complement $\mathbb{R}\cos(\cdot)$. Thence,
\begin{align*}
X = X^\perp \oplus \mathbb{R}\cos(\cdot) \quad \text{and} \quad Y = Y^\perp \oplus \mathbb{R}\cos(\cdot).
\end{align*}
\subsection{Kernel and range of the linear operator \texorpdfstring{$L_{\lambda_*}$}{L-lambda*}}
In the sequel, $\gamma\in(0,\gamma_*)$ is fixed and $\lambda_*= \lambda_*(\gamma)$ is the unique $\lambda>0$ as in Theorem \ref{thm:Lambda-star}. We want to characterize of the kernel and the range of the operator $L_{\lambda_*}$.
We will need the following proposition.

\begin{proposition}\label{pro:F-in-Holder}
Let $w\in C^{1,0}(\mathbb{R})$ and consider the  function $F:\R\to\R$ with
\begin{align*}
F(t):=\int_{\mathbb{R}}\int_{\mathbb{S}^1}(w(t)-w(t-r)) G_\kappa\big((r^2+\lambda ^2_*|e_1-\theta|^2)^{1/2} \big)\mathrm{d}r \mathrm{d}\theta.
\end{align*}
Then  $F$ belongs to the log-H\"older space $C^{0, \log}(\R)$, that is, $F$ is bounded and satisfies the following log-H\"older condition:
\begin{align*}
|F(t_1)-F(t_2)|\leq C| t_1-t_2|\big|\log(| t_1-t_2|) \big|\,  \|w'\|_{L^\infty(\R)} \qquad\textrm{for \,\,$|t_1-t_2| \leq e^{-1}$},
\end{align*}
for some generic constant $C$. Moreover,  for every $\alpha\in[0,1)$  we have $F\in C^{0,\alpha}(\R)$ and
\begin{align*}
\|F\|_{C^{0,\alpha}(\R)}\leq C  \|w'\|_{L^\infty(\R)}.
\end{align*}
\end{proposition}
\begin{proof}
 First of all, since  $r\mapsto G_\kappa (r)=r^{-1}e^{-\kappa r}$
is  decreasing on $(0,\infty)$ we get
\begin{align*}
 G_\kappa\left((r^2+\lambda ^2_*|e_1-\theta|^2)^{1/2}\right) \leq G_\kappa(|r|).
\end{align*}
Meanwhile, since  $w\in C^{1,0}(\mathbb{R})$, we find that
\begin{align}\label{eq:estima-in-Holder1}
|w(t)-w(t-r)|=\Big|r \int_{0}^1  w'(  t-\varrho r) d\varrho\Big|\leq  |r| \|w'\|_{L^\infty(\R)}.
\end{align}
This together with the previous estimate imply that $F$ is bounded  with
\begin{align}\label{eq:F-in-L-infty}
|F(t)|
&\leq |\mathbb{S}^1|\|w'\|_{L^\infty(\R)} \int_{\mathbb{R}} \exp(-\kappa|r|)\mathrm{d}r =  4\pi \kappa^{-1}\|w'\|_{L^\infty(\R)}.
\end{align}
In particular this implies
\begin{align}\label{eq:F-holder-infty}
|F(t_1)-F(t_2)|\leq 8\pi  e^\alpha\|w'\|_{L^\infty(\R)} |t_1-t_2|^\alpha\qquad\textrm{for $|t_1-t_2|\geq e^{-1}$}.
\end{align}
Now, assume  $a:=|t_1-t_2|\leq e^{-1}$ so that $-\log(a)\geq 1$.  For $t_1,t_2\in \R$, we get
\begin{align*}
|(w(t_1)-w(t_1-r))- (w(t_2)-w(t_2-r) )|
&=  \Big|\Big(\int_{t_2}^{t_1} \hspace*{-1ex}- \int_{t_2-r}^{t_1-r}\hspace*{-0.5ex}\Big) w'(t)\mathrm{d}t\Big|\\
&\leq   2 |t_1-t_2|  \|w'\|_{L^\infty(\R)}.\nonumber
\end{align*}
Combining this with the estimate \eqref{eq:estima-in-Holder1},  and $a=|t_1-t_2|$ we obtain
\begin{align*}
|[w(t_1)-w(t_1-r)]- [w(t_2)-w(t_2-r) ] | \leq  2 \|w'\|_{L^\infty(\R)} \min (a,|r|).
\end{align*}
From this and the fact that $-\log(a)\geq 1$, we get
\begin{align*}
|F(t_1)-F(t_2)|&\leq 4\pi  \|w'\|_{L^\infty(\R)} \int_{\mathbb{R}} \min (a,|r|)    |r|^{-1} e^{-\kappa |r|}\mathrm{d}r \\
&=8\pi\|w'\|_{L^\infty(\R)}  \left( \int_{0}^{a}  e^{-\kappa r} \mathrm{d}r + a\int_{a}^{1} r^{-1}   e^{-\kappa r}\mathrm{d}r +a\int_{ 1}^{\infty}\hspace{-1ex} r^{-1} e^{-\kappa r}\mathrm{d}r\right)\\
& \leq 8\pi\|w'\|_{L^\infty(\R)}  \left( \int_{0}^{a}  \mathrm{d}r + a\int_{a}^{1} r^{-1}   \mathrm{d}r +a\int_{ 1}^{\infty} e^{-\kappa r} \mathrm{d}r\right)\\
&= 8\pi\|w'\|_{L^\infty(\R)} (-a\log(a)+ a(1+\kappa^{-1}e^{-\kappa})) \\
&\leq -a\log(a) 8\pi(2+\kappa^{-1}e^{-\kappa}) \|w'\|_{L^\infty(\R)}.
\end{align*}
That is $F$ satisfies the log-H\"older condition
\begin{align*}
|F(t_1)-F(t_2)|\leq C| t_1-t_2|\big|\log(| t_1-t_2|) \big|\,  \|w'\|_{L^\infty(\R)} \qquad\textrm{for \,\,$|t_1-t_2| \leq e^{-1}$}
\end{align*}
with $C=8\pi(2+\kappa^{-1}e^{-\kappa})$. Since $\max_{x\in [0,1]}(-x \log x) = e^{-1}$, substituting $x = a^{1-\alpha}$ for $\alpha \in [0,1)$ and $a \in [0,1]$ gives $-a\log a \leq \frac{a^{\alpha}}{(1-\alpha)e}$. Accordingly,  we deduce that
\begin{align*}
|F(t_1)-F(t_2)|\leq \frac{8\pi(2+\kappa^{-1}e^{-\kappa})}{(1-\alpha)e}| t_1-t_2|^\alpha\,  \|w'\|_{L^\infty(\R)} \qquad\textrm{for \,\,$|t_1-t_2| \leq e^{-1}$}.
\end{align*}
Altogether with the estimates \eqref{eq:F-in-L-infty} and  \eqref{eq:F-holder-infty}, we obtain $F\in C^{0,\alpha}(\R)$ with
\begin{align*}
|F(t_1)-F(t_2)|&\leq C_\alpha\|w'\|_{L^\infty(\R)} |t_1-t_2|^\alpha\quad \text{and}\quad
|F(t_1)|\leq C_\alpha\|w'\|_{L^\infty(\R)} \quad\quad\textrm{for\,\,  $t_1,t_2\in \R$},
\end{align*}
with $C_\alpha=\max\big(8\pi  e^\alpha, \tfrac{8\pi(2+\kappa^{-1}e^{-\kappa})}{(1-\alpha)e}\big)$. Hence we have $\|F\|_{C^{0,\alpha}} \leq 2C_\alpha\|w'\|_{L^\infty(\R)}$.
\end{proof}

\noindent The most important feature  of spaces $X^\perp$ and $Y^\perp$ is given by  the following theorem.
\begin{theorem}\label{thm:rang-L-star}
Let $\gamma\in (0,\gamma_*)$ and $\lambda_{*}=\lambda_*(\gamma)$ be as in Theorem \ref{thm:Lambda-star}.  For every $f\in Y^\perp$ there exists a unique $w\in X^\perp$ such that
\begin{align}\label{eq:L-star-w-eq-f-cyl}
L_{\lambda_*} (w)=f \quad \text{ on  $~\mathbb{R}$}.
\end{align}
In addition, the restriction $L_{\lambda_*} : X^\perp \to Y^\perp$ is bijective with range $R(L_{\lambda_*}) = Y^\perp$.
\end{theorem}
\begin{proof}
This proof is subdivided into three main steps.

\textit{Step 1. Existence and uniqueness}.
Note that for $w\in X^\perp$ and  $f\in Y^\perp$, $w$ and $f$ are periodic and even with  $ f_1= \int_0^{2\pi}f(t)\cos(t)\, dt=0 $ and $w_1 =\int_0^{2\pi}w(t)\cos(t)\mathrm{d}t= 0.$
The decomposition in  Fourier series, see for instance \cite{Duo01},  infers that
\begin{align*}
f(t)&=f_0+ \sum_{k=2}^\infty f_k \cos( k t),
\quad \text{$f_k =\frac{1}{\pi}\int_0^{2\pi} \hspace*{-1ex}f(t)\cos(kt)\mathrm{d}t$ and $f_0 =\frac{1}{2\pi}\int_0^{2\pi}\hspace*{-1ex}f(t) \mathrm{d}t$},\\
w(t)&=w_0+ \sum_{k=2}^\infty w_k \cos( k t), \quad\text{$w_k =\frac{1}{\pi}\int_0^{2\pi} \hspace*{-1ex}w(t)\cos(kt) \mathrm{d}t$ and $w_0 =\frac{1}{2\pi}\int_0^{2\pi} \hspace*{-1ex}w(t)\mathrm{d}t$}.
\end{align*}
 Given $f$, we seek $w$ such that $L_{\lambda_*}(w) = f$. However, by Lemma \ref{lem:def-linearized-eigen} we  have, $ L_{\lambda_*}( \cos( k t))=  \sigma_{\lambda _*,\gamma}(k) \cos( k t)$, thus the continuity of the linear operator $L_{\lambda_*}$, implies that
\begin{align}\label{eq:Llambda-fourier}
L_{\lambda_*}(w)(t)=\sigma_{\lambda _*,\gamma}(0)w_0+ \sum_{k=2}^\infty w_k \sigma_{\lambda _*,\gamma}(k)\cos(k t),
\end{align}
It follows that $L_{\lambda_*} (w)=f$ if and only if
\begin{align*}
(\sigma_{\lambda _*,\gamma}(0)w_0-f_0)+ \sum_{k=2}^\infty\left( w_k \sigma_{\lambda _*,\gamma}(k)-f_k \right)\cos( k t) =0,
\end{align*}
with $w_1 =\sigma_{\lambda _*,\gamma}(1)=f_1 =  0$. Because $\sigma_{\lambda_*,\gamma}(k) \neq 0$ for all $k \neq 1$ (by \eqref{eq:sk-neq-0}), the uniqueness of the Fourier decomposition implies that
\begin{align*}
w_k= \frac{1}{\sigma_{\lambda _*,\gamma}(k)} f_k \qquad \textrm{ for every $k\neq 1$ and $w_1 = 0$}.
\end{align*}
Finally, $w$ is even, $2\pi$-periodic, $  \int_0^{2\pi}w(t)\cos(t)\, dt=0 $. Moreover, the Fourier decomposition of $f$ guarantees the uniqueness of $w$ satisfying $L_{\lambda_*} (w)=f $ with
\begin{align}\label{eq:express-w}
w(t)=\frac{ f_0}{\sigma_{\lambda _*,\gamma}(0)} + \sum_{k=2}^\infty \frac{ f_k}{\sigma_{\lambda _*,\gamma}(k)} \cos( k t).
\end{align}

\textit{Step 2.  Sobolev regularity: $w\in H^2_{\mathrm{loc}}(\R)\cap  W^{1,\infty}(\R)$}. Since the functions $t\mapsto \cos(kt )$ are smooth, the weak derivatives of $w$ are given as follows
\begin{align*}
w'(t) &= -\sum_{k=2}^\infty \frac{ kf_k}{\sigma_{\lambda _*,\gamma}(k)} \sin( k t)
\quad \text{and}\quad
w''(t)=- \sum_{k=2}^\infty \frac{k^2  f_k}{\sigma_{\lambda _*,\gamma}(k)} \cos( k t).
\end{align*}
Since $\sigma_{\lambda _*,\gamma}(k)\sim \frac{1}{k^2}$ as $k\to \infty$,  by   \eqref{eq:eigen-p3}, it follows that the sequences$ \frac{ 1}{\sigma_{\lambda _*,\gamma}(k)},~~\frac{ k}{\sigma_{\lambda _*,\gamma}(k)}$ and  $\frac{ k^2}{\sigma_{\lambda _*,\gamma}(k)} $ are
bounded,  while the following series converge
\begin{align*}
\sum_{k=2}^\infty \frac{1}{\sigma^2_{\lambda _*,\gamma}(k)} \qquad\textrm{and}\qquad\sum_{k=2}^\infty \frac{ k^2}{\sigma^2_{\lambda _*,\gamma}(k)}.
\end{align*}
Note that $f\in Y^\perp\subset C^{ 0,\alpha}(\R)$ so that $f\in L^2(0,2\pi)$. By Parseval's identity  we get
\begin{align*}
&\|w^{(j)}\|^2_{L^2(0,2\pi)} =\sum_{k=0, \,k\neq 1}^{\infty} \frac{ k^{2j}f_k^2 }{\sigma^2_{\lambda _*,\gamma}(k)} \leq  C\sum_{k=0}^{\infty} f_k^2 =  C\|f\|^2_{L^2(0,2\pi)}\qquad\, j=0,1,2.
\end{align*}
 Accordingly, $w^{(j)} \in L^2(0, 2\pi)$ and hence $w \in H^2(0, 2\pi)$. By periodicity, we have $w \in H^2_{\mathrm{loc}}(\R)$. Meanwhile, Cauchy-Schwarz inequality implies that $w \in W^{1,\infty}(\R)$ since
\begin{align*}
\|w\|^2_{L^\infty(\R)}+\|w'\|^2_{L^\infty(\R)}
&\leq 2\|f\|^2_{L^2(0,2\pi)}\, \sum_{k=2}^\infty \frac{k^2+1}{\sigma^2_{\lambda _*,\gamma}(k)} <\infty.
\end{align*}

 \textit{Step 3. H{\"o}lder regularity: $w\in C^{2,\alpha} (\R)$}. Morrey's embedding, see, e.g., \cite[Theorem 12.55]{Leo17} or \cite[Corollary 9.13]{Bre11}, yields $H^2_{\mathrm{loc}}(\R) \subset C^{1, 1/2}_{\mathrm{loc}}(\R)$. Using H\"{o}lder injection   $C^{1,1/2}_{\mathrm{loc}}(\R)\subset C^{1,0}_{\mathrm{loc}}(\R)$  see for instance \cite[chapter 1]{AdFo03}.
Therefore $w\in C^{1,0}_{\mathrm{loc}}(\R)$ and hence  $w'$ is the derivative in the classical sense.
Moreover, since $w\in W^{1,\infty}(\R)$, i.e., $ w,w'\in L^\infty(\R^d)$ it follows that $w\in C^{1,0}(\R)$.  Let us recall that from the expression \eqref{eq:reg-func-cyliner} we have
\begin{align}\label{eq:lambdastart-F}
L_{\lambda_*}(w)& =-w''- \frac{1}{\lambda_*^2} w+ \gamma  \lambda_*  T_2^0(\lambda_*\kappa)w-
\gamma  \lambda_*  F,
\end{align}
where  $F$ given in Proposition \ref{pro:F-in-Holder}  while $T_2^0(\lambda_*\kappa)$ is  given through  \eqref{eq:I2} and
\eqref{eq:Bessel-Eiegn9}, that is,
\begin{align*}
F(t)&=\int_{\mathbb{R}}\int_{\mathbb{S}^1}(w(t)-w(t-r)) G_\kappa\big((r^2+\lambda ^2_*|e_1-\theta|^2)^{1/2} \big)\mathrm{d}r \mathrm{d}\theta,\\
T_2^0(\lambda_*\kappa)&=\frac{1}{2}\int_{\mathbb{R}}\int_{\mathbb{S}^1}|\theta- e_1|^2  G_\kappa\left((r^2+\lambda_*^2|e_1-\theta|^2)^{1/2} \right)\mathrm{d}r \mathrm{d} \theta.
\end{align*}
Therefore, letting $W= w-w(0)$,  the equation $L_{\lambda_*}(w) = f$ with $C^{0,\alpha}(\R)$ infers
\begin{align*}
\begin{cases}
- W''+ \big( \gamma \lambda_*   T_2^0(\lambda_*\kappa)- \frac{1}{\lambda_*^2}\big)W = \gamma\lambda_* F+ f+c &\text{ in } \,\,(0,2\ell\pi),\\
\quad W(0)= W(2\ell\pi)= 0,&
\end{cases}
\end{align*}
with $\ell\in \mathbb{N}$,  $c= -( \gamma \lambda_*   T_2^0(\lambda_*\kappa)- \frac{1}{\lambda_*^2})w(0)$, $w\in C^{1,0}(\R)$, $f\in C^{0,\alpha}(\R)$ and  $F\in C^{0,\alpha}(\R)$ by  Proposition \ref{pro:F-in-Holder}. That is $\gamma\lambda_* F+ f\in C^{0,\alpha}(\R)$.
 From standard Schauder estimates in elliptic regularity theory, see \cite[Theorem 9.33]{Bre11}, we obtain $W \in C^{2,\alpha}(0,2\ell\pi)$. Hence, by periodicity, $w \in C^{2,\alpha}_{\mathrm{loc}}(\mathbb{R})$. However, using \eqref{eq:lambdastart-F}, we find that $w''$ is also bounded on $\R$, since $f$, $F$, and $w$ are bounded. This boundedness, combined with the fact that $w'' \in C^{0,\alpha}_{\mathrm{loc}}(\R)$, implies that $w'' \in C^{0,\alpha}(\R)$. This together with the fact that $w\in C^{1,0}(\R)$, we obtain $w\in C^{2,\alpha}(\R)$.  Consequently, because $w$ is $2\pi$-periodic, even, $\int_0^{2\pi} w(t)\cos(t)\,\mathrm{d}t = 0$, we conclude there exists a unique $w \in X^\perp$ given by \eqref{eq:express-w} such that  $L_{\lambda_*}(w) = f$.

From this we conclude  that $Y^\perp \subset R(L_{\lambda_*})$.
Conversely, for $w\in X^\perp\subset C^{2,\alpha}(\R)$ we have $w,w''\in C^{0,\alpha}(\R)$, while
$F\in C^{0,\alpha}(\R)$ by Proposition \ref{pro:F-in-Holder}. Hence, $L_{\lambda_*}(w) =-w''- \frac{1}{\lambda_*^2} w+ \gamma  \lambda_*  T_2^0(\lambda_*\kappa)w-
\gamma  \lambda_*  F$ also belongs to $C^{0,\alpha}(\R)$.
Moreover, since $w$ is even, $2\pi$-periodic and $\int_0^{2\pi} w(t)\cos(t)\mathrm{d}t=0$, the Fourier decomposition of $L_{\lambda_*}(w)$ takes the form \eqref{eq:Llambda-fourier}. The latter expression clearly shows that $L_{\lambda_*}(w)$ is even, $2\pi$-periodic and $\int_0^{2\pi} L_{\lambda_*}(w)(t)\cos(t)\mathrm{d}t=0$. Thus $L_{\lambda_*}(w)\in Y^\perp$. We deduce that $L_{\lambda_*}$ is a bijection and
$R(L_{\lambda_*})=Y^\perp $.
\end{proof}


\subsection{Proof of Theorem \ref{thm:main-thm}}
To solve the equation  \eqref{eq:new-equation-to solve} using Crandall-Rabinowitz Theorem \ref{thm:crandall}, we fix $\gamma\in (0,\gamma_*)$ and $\lambda _*=\lambda _*(\gamma)$ as before and  we define the function
\begin{align*}
\mathcal{G}_\gamma:  (-\lambda_*,\lambda_* )
\times X \to Y, \quad \mathcal{G}_\gamma(\nu ,u) :=\mathcal{F}(\nu +\lambda_*+ u)
-\mathcal{ F}(\nu+\lambda_*).
\end{align*}
To solve the equation $\mathcal{G}_\gamma(\nu, u) = 0$, we first collect the results established thus far.
\begin{itemize}
\item[(a)]  Clearly,  $\mathcal{G}_\gamma(\nu ,0)= \mathcal{G}_\gamma(0 ,0)=0 $,  for every $\nu \in (-\lambda_*,\lambda_* ).$
 The smoothness of  $\cF$, see Theorem \ref{thm:diff-cF},  readily implies the smoothness  of $\cG_\gamma$.
In particular,  we  have
\begin{align*}
\quad \quad  D_u \cG_\gamma(\nu,u)= D\cF(\nu+\lambda_*+ u)
\, \quad  \text{and}\, \quad
\partial_\nu D_u\, \cG_\gamma(\nu,u)(\cdot) =D^2\cF(\nu+
\lambda_* +u)(\cdot,1).
\end{align*}
 \item[(b)] In particular, we find that $D_u \cG_\gamma(\nu,0) = D\cF(\nu+\lambda_*) =L_{\nu+\lambda_*}$  and
\begin{align*}
D_u\mathcal{G}_\gamma(0,0) = D\mathcal{F}(\lambda_*) = L_{\lambda_*}.
\end{align*}
According to Theorem \ref{thm:rang-L-star} we have
$$R(L_{\lambda_*})= R(D_u \mathcal{G}_\gamma(0,0 ))= Y^\perp.$$
We also know that,  $Y= Y^\perp\oplus \R\cos(\cdot)$, that is, $Y^\perp$ is a complement of $ \mathbb{R}\cos(\cdot)$.  Namely $\textrm{codim}R(D_u \mathcal{G}_\gamma(0,0 ))=1 $. Altogether  with Corollary \ref{cor:eigenfunct-of-Llam} we find that
 $$ N(D_u \mathcal{G}_\gamma(0,0 ))=Y\backslash R(D_u \mathcal{G}_\gamma(0,0 ))= \mathbb{R} \cos(\cdot).$$

\item[(c)]
By Lemma \ref{lem:def-linearized-eigen}, $\sigma_{\nu + \lambda_*,\gamma}(1)$ is an eigenvalue associated to eigenfunction $\cos(\cdot)$, that is
$$D_u \mathcal{G}_\gamma(\nu,0 )(\cos(\cdot))= L_{\nu + \lambda_*}(\cos(\cdot))= \sigma_{\nu + \lambda_*,\gamma}(1 )\cos(\cdot).$$
Since $\cG_\gamma$ is $C^\infty$ with respect to the H\"older norm, we can compute the pointwise directional derivative
\begin{align*}
\partial_\nu D_u\, \cG_\gamma(0,0)(\cos(\cdot))
&= \lim_{\nu \to 0}\frac{D_u\, \cG_\gamma(\nu,0)-D_u\, \cG_\gamma(0,0)  }{\nu}(\cos(\cdot))\\
&=  \lim_{\nu \to 0}\frac{L_{\nu+\lambda_*} (\cos(\cdot))-L_{\lambda_*}(\cos(\cdot))  }{\nu}
\\&= \partial_\nu [\sigma_{\nu + \lambda_*,\gamma}(1 )]\cos(\cdot) |_{\nu=0}
= \partial_\nu\sigma_{\lambda_*,\gamma}(1) \cos(\cdot).
\end{align*}
By  Lemma \ref{lem:monoton-lambda}, we have
$d_*: =\partial_\nu \sigma_{\lambda_*,\gamma}(1) >0, $ then
$$ \partial_\nu D_u \mathcal{G}_\gamma(0,0 )[\cos(\cdot)]= d_*  \cos(\cdot) \notin Y^\perp=  R(D_u \mathcal{G}_\gamma
(0 ,0 )).$$
\end{itemize}
By the Crandall-Rabinowitz Theorem \ref{thm:crandall} (see \cite[Theorem 1.7]{CrRa71}), we then find ${\varepsilon_0}>0$ and smooth curves
\begin{align*}
&\nu :(-\varepsilon_0, \varepsilon_0) \to (0,\infty)\quad s \mapsto \nu_s
\\ 
&\omega :(-\varepsilon_0, \varepsilon_0) \to X^\perp \quad s \mapsto \omega_s.
\end{align*}
such that $\nu_0= 0,$ $\omega_0\equiv 0$ and   setting $u_s = s (\cos(\cdot) + \omega_s)$ we have
\begin{align*}
\cG_\gamma(\nu_s, u_s) = \cF(\nu_s+\lambda_*+u_s) - \cF(\nu_s+\lambda_*)=0 \quad\text{for all}\quad   s \in (-{\varepsilon_0},{\varepsilon_0})
\end{align*}
 In addition, the following hold
\begin{enumerate}
\item Since $\omega_s\in X^\perp$, $\omega_s$ is even, $2\pi$-periodic, $\omega_s\in C^{2,\alpha}(\R)$ and
\begin{align*}
\int_0^{2\pi} \omega_s(t)\cos(t)\mathrm{d}t=0.
\end{align*}
\item
Letting $\lambda_s =  \nu_s+\lambda_*$ and  $\varphi_s=\lambda_s + u_s $ one finds that, $\lambda_0=\lambda_*$, $u_0\equiv \lambda_*$, $\varphi_0\equiv \lambda_*$  and
\begin{align*}
\cG_\gamma(\nu_s, u_s) = \cF(\varphi_s) - \cF(\lambda_s)=0 \quad\text{for all}\quad   s \in (-{\varepsilon_0},{\varepsilon_0}).
\end{align*}
\item
By continuity, $\lim_{s \to 0} \varphi_s = \lambda_* > 0$ in $C^{2,\alpha}(\mathbb{R})$, and  choosing $\varepsilon_0 > 0$ sufficiently small up to a relabeling yields
$$\|\varphi_s- \lambda_*\|_{L^\infty(\R)} \leq \|\varphi_s- \lambda_*\|_{C^{2,\alpha}(\R)} < \frac{\lambda_*}{2}$$ for all $s \in (-\varepsilon_0, \varepsilon_0)$, which implies $\varphi_s > \frac{\lambda_*}{2} > 0$ for every $s \in (-\varepsilon_0, \varepsilon_0)$.
\item  Therefore, $\varphi_s\in \mathcal{O}\cap X$ satisfies  $ \cF(\varphi_s) = C_{\lambda_s,\kappa}$, where   the constant $C_{\lambda_s,\kappa}:=\cF(\lambda_s)$ is given as in Proposition \ref{prop:var-functionals} and clearly verifies $C_{\lambda_s,\kappa}\leq C_{\lambda_*,\kappa}$.
\item
From this result we conclude that, the family of unbounded domains $(\Omega_{\varphi_s})_{s\in (-\varepsilon_0,\varepsilon_0)}$ with   $\varphi_s\in \cO\cap X$ are stationary sets for the functional energy \eqref{eq:Geom-pblem-interface-like-s-perim} or solutions of the geometrical  variational problem \eqref{eq:main-problem}.
\end{enumerate}
\begin{remark}
The Fourier series decomposition suggests that each $\omega_s$ is of the form
\begin{align*}
\omega_s(t)= a_{s,0}+ \sum_{k=2}^N a_{s,k} \cos(kt)\qquad a_{s,k}\in \R, \quad N\geq 2.
\end{align*}
\end{remark}
\section{Proof of Theorem \ref{thm:diff-cF}: Differentiability of  \texorpdfstring{$\cF$}{F}}
\label{sec:different-F}
In this section we only prove Theorem \ref{thm:diff-cF}, that is,  smoothness of $\mathcal{F}: \cO\cap X\to Y$ with $\mathcal{F}(\varphi)=\cF_0(\varphi)+\gamma\cF_1(\varphi)$.
Note that for any $\varphi \in \mathcal{O} \cap X$, there exists a $\delta > 0$ such that $\varphi \ge \delta$. In this section, we fix $\delta>0$ and $\varphi \in \mathcal{B}_\delta$, where $\mathcal{B}_\delta := \{\varphi \in C^{2,\alpha}(\R) : \varphi > \delta\}$, since $\mathcal{B}_\delta$ is an open subset of $C^{2,\alpha}$ and is therefore suitable for differentiation.
\subsection{Regularity of \texorpdfstring{$\mathcal{F}_0$}{F0}}
The maps $x \mapsto (1+x^2)^{-a}$, $a\in \{\frac{1}{2}, \frac{3}{2}\}$, and $x \mapsto \frac{1}{x}$ for $x > \delta$ are $C^\infty$ (and even analytic). Meanwhile, the continuous linear\footnote{Every continuous linear map is smooth with higher-order derivatives vanishing identically.}
 maps $L_j: C^{2,\alpha}(\R) \to C^{0,\alpha}(\R)$ defined by $w \mapsto w^{(j)}$ for $j = 0, 1, 2$ are $C^\infty$ and satisfy
$\|L_j(w)\|_{C^{0,\alpha}} \leq \|w\|_{C^{2,\alpha}}.$
 This implies that $\cF_0: \mathcal{B}_\delta \to C^{0,\alpha}$ is $C^\infty$, where we recall
\begin{align*}
\cF_0(\varphi)(t)&:= \frac{-\varphi''}{(1+\varphi'^2)^{\frac{3}{2}}}+\frac{1}{\varphi(1+\varphi'^2)^{\frac{1}{2}}}.
\end{align*}
Now we perform the Taylor expansion. First, we write
\begin{align*}
1+(\varphi'+\varepsilon w')^2 = A + \varepsilon B + \varepsilon^2 C \quad \text{where} \quad A = 1+\varphi'^2,\,\,  \; B = 2\varphi' w', \,\, \,  C = (w')^2.
\end{align*}
Now substituting $\tau= \varepsilon B + \varepsilon^2 C$, with the aid of Taylor expansions
formula \begin{align*}
(1+x)^\beta = 1 + \beta x + \frac{\beta(\beta-1)}{2!} x^2 + \frac{\beta(\beta-1)(\beta-2)}{3!} x^3 + \dots \quad |x|<1,\,\, \beta\in \R,
\end{align*}
we find that
\begin{align*}
(A+\tau)^{-1/2}
&=A^{-1/2} - \varepsilon \frac{\varphi' w'}{A^{3/2}}  + \varepsilon^2 \left( -\frac{(w')^2}{2A^{3/2}} + \frac{3\varphi'^2(w')^2}{2A^{5/2}} \right) + O((\varepsilon w')^3),\\
(A+\tau)^{-3/2}
&= A^{-3/2} - 3\varepsilon \frac{\varphi' w'}{A^{5/2}} + \varepsilon^2 \left( -\frac{3}{2}\frac{(w')^2}{A^{5/2}} + \frac{15}{2}\frac{\varphi'^2(w')^2}{A^{7/2}} \right) + O((\varepsilon w')^3).
\end{align*}
 In addition, we have
\begin{align*}
(\varphi+\varepsilon w)'' &= \varphi'' + \varepsilon w''
\qquad
\text{and}\qquad
 \frac{1}{\varphi+\varepsilon w}
= \frac{1}{\varphi} - \varepsilon\frac{w}{\varphi^2} + \varepsilon^2\frac{w^2}{\varphi^3}
+ O((\varepsilon w)^3).
\end{align*}
Altogether, we find that
\begin{align*}
\cF_0(\varphi+\varepsilon w) &=
-\frac{\varphi''+\varepsilon w''}{(1+(\varphi'+\varepsilon w')^2)^{3/2}} +\frac{1}{\varphi+\varepsilon w} \frac{1}{(1+(\varphi'+\varepsilon w')^2)^{1/2}}\\
&=\cF_0(\varphi) + \varepsilon L_\varphi(w) + \varepsilon^2 B_\varphi(w,w) + \varepsilon^3 O(R_\varphi(w,w',w'')),
\end{align*}
where the linear and the quadratics  $L_\varphi(w),B_\varphi(w,w):C^{2,\alpha}(\R )\to C^{0,\alpha}(\R)$ with

\begin{align}
\label{eq:diff-first-f0}
L_\varphi(w)
&:=  -\frac{w''}{(1+\varphi'^2)^{3/2}} + \frac{3\varphi''\varphi'w'}{(1+\varphi'^2)^{5/2}} -\frac{\varphi'w'}{\varphi (1+\varphi'^2)^{3/2}} -\frac{w}{\varphi^2(1+\varphi'^2)^{1/2}}\\
 \begin{split}\label{eq:diff-second-f0}
B_\varphi(w,w)&:=3\frac{\varphi'w'w''}{A^{5/2}}+
\frac32\frac{\varphi''(w')^2}{A^{5/2}}
- \frac{15}{2}
\frac{\varphi''\varphi'^2(w')^2}{A^{7/2}} \\
&\quad\quad-
\frac{(w')^2}{2\varphi A^{3/2}} +
\frac{3\varphi'^2(w')^2}{2\varphi A^{5/2}}+
\frac{\varphi'ww'}{\varphi^2A^{3/2}}+ \frac{w^2}{\varphi^3A^{1/2}},\qquad (A:= 1+\varphi'^2),
\end{split}
\end{align}
 are  bounded operators, as  it is not difficult to check that
\begin{align*}
\|L_\varphi(w)\|_{C^{2,\alpha}(\R)}\leq C_\varphi\|w\|_{C^{0,\alpha}(\R)}\quad \text{and}\quad \|B_\varphi(w,w)\|_{C^{2,\alpha}(\R)}\leq C_\varphi\|w\|^2_{C^{0,\alpha}(\R)}.
\end{align*}
 Moreover, the remainder $R_\varphi(w, w',w'')$ is a homogeneous  polynomial of degree $3$ in variable $w, w'$ and $w''$.  So that for a generic constant only depending on $C_\varphi>0$ we have
\begin{align*}
\|R_\varphi(w, w',w'')\|_{C^{0,\delta}(\R)}\leq C_\varphi\|w\|^3_{ C^{2,\alpha}(\R)}.
\end{align*}
 This implies that, taking $\varepsilon=1$ we find that
\begin{align*}
\lim_{\|w\|_{C^{2,\alpha}(\mathbb{R})}\to 0} \frac{\|\mathcal{ F}_0(\varphi+ w)-\mathcal{ F}_0(\varphi)- L_\varphi(w) -B_\varphi(w,w) \|_{C^{0,\alpha}(\mathbb{R})}}{\|w\|^2_{C^{2,\alpha}(\mathbb{R})}} = 0.
\end{align*}
 Therefore, it follows that $D\cF_0(\varphi)(w)=L_\varphi(w)$ and $D^2\cF_0(\varphi)(w,w)=2B_\varphi(w,w)$.
From the foregoing we deduce the following result.
\begin{lemma}\label{lem:diff-f0}
The mapping $\cF_0: \mathcal{B}_\delta \to C^{0,\alpha}$ is Fr\'echet $C^\infty$. Moreover, for $\varphi\in  \mathcal{B}_\delta$, the first derivative the linear map  $D\cF_0(\varphi) : C^{2,\alpha}(\R )\to C^{0,\alpha}(\R)$  is given
$$D\cF(\varphi)(w)= L_\varphi(w)  =-\frac{w''}{(1+\varphi'^2)^{3/2}} + \frac{3\varphi''\varphi'w'}{(1+\varphi'^2)^{5/2}} -\frac{\varphi'w'}{\varphi (1+\varphi'^2)^{3/2}} -\frac{w}{\varphi^2(1+\varphi'^2)^{1/2}},$$
while the  Hessian $D^2\cF_0(\varphi):C^{2,\alpha}(\R )\times C^{2,\alpha}(\R )\to C^{0,\alpha}(\R) $ is  given by polarization
\begin{align*}
D^2\cF_0(\varphi)(w,v) =
\left(B_\varphi(w+v,w+v) -B_\varphi(w,w)-B_\varphi(v,v)
\right).
\end{align*}
\end{lemma}
\subsection{Regularity of \texorpdfstring{$\cF_1$}{F1}}
We will first look for a "candidate" for the derivative of $\cF_1$. Then we prove the smoothness of this formal derivative of $\cF_1$. This is the program that will be carried over in   the renaming of this section.
\subsubsection{Candidate for the derivative of $\cF_1$} \label{ss:can-deriv}
In this paragraph, we heuristically derive the expression for the derivative of $\cF_1$ (see \eqref{eq:DefcF1} below), where
\begin{align*}
\cF_1(\varphi)(t)
&=  \int_{\mathbb{R}}\int_{B_1 }\varphi^2(s)
G_\kappa(|(t, \varphi(t) e_1)-(s,\varphi(s)y)|) \mathrm{d} y \mathrm{d}s.
\end{align*}
We observe that for fixed $t,s$, the function
\begin{align*}
f(y)=G_\kappa(|(t, \varphi(t) e_1)-(s,\varphi(s)y)|) =  G_\kappa\big(\left((t-s)^2 + |\varphi(s)y - \varphi(t)e_1|^2\right)^{1/2}\big)
\end{align*}
satisfies, when heuristically differentiating with respect to $\varphi$ in Fr\'echet,

\begin{align*}
\varphi^2(s) & D_\varphi f(w) (y)
 =  \varphi^2(s)G'_\kappa(\cdots) \times
\left\lbrace\frac{(w(s)y-w(t)e_1)\cdot (\varphi(s)y- \varphi(t) e_1) }{\left((t-s)^2 + |\varphi(s)y - \varphi(t)e_1|^2\right)^{1/2} } \right\rbrace\\
&= \varphi^2(s)G_\kappa'(\cdots)\times  \left\lbrace (w(s)y-w(t)e_1)\cdot \frac{\nabla_y|(t, \varphi(t) e_1)-(s,\varphi(s)y)| }{\varphi(s)} \right\rbrace\\
&=\varphi(s)(w(s)y-w(t)e_1)\cdot \nabla_y \left[G_\kappa (|(t, \varphi(t) e_1)-(s,\varphi(s)y)|)\right]\\
&= \varphi(s)(w(s)y-w(t)e_1)\cdot \nabla f(y).
\end{align*}
Therefore, a formal computation, of $D\cF_1(\varphi)(w)$  via  chain rule reveals that
\begin{align}\label{eq:DefcF1-heur}
\begin{split}
D\cF_1(\varphi)(w)(t)
&=\int_{\R}\int_{B_1} \hspace*{-1ex}2\varphi(s)w(s) f(y)\mathrm{d}y \mathrm{d}s + \varphi(s)  (w(s)y-w(t)e_1)\cdot \nabla f(y)\mathrm{d}y \mathrm{d}s .
\end{split}
\end{align}
Applying integration by parts to the second term gives
\begin{align*}
\int_{B_1}\hspace{-1ex}  (w(s)y-w(t)&e_1)\cdot \nabla f(y)\mathrm{d}y
=-2 \int_{B_1} \hspace{-1ex} w(s) f(y)\mathrm{d}y+ \int_{\mathbb{S}^1}\hspace{-1ex}  (w(s)\theta-w(t)e_1)\cdot  \theta f(\theta)\mathrm{d}\theta\\
&= -2 \int_{B_1} \hspace{-1ex} w(s) f(y)\mathrm{d}y + \int_{\mathbb{S}^1} \Big[(w(s)-w(t)) + \frac{w(t)}{2}|\theta-e_1|^2\Big] f(\theta)\mathrm{d}\theta
\end{align*}
where  $(1-\theta\cdot e_1) =\frac12|\theta-e_1|^2$. Inserting this in \eqref{eq:DefcF1-heur}, we obtain
\begin{align}
\label{eq:DefcF1}
\begin{split}
D\cF_1(\varphi)(w)(t)
&=  \int_\R \int_{\mathbb{S}^1}(w(s)-w(t))   G_\kappa(|(t, \varphi(t) e_1)-(s,\varphi(s)\theta)|) \varphi(s) \mathrm{d}\theta  \mathrm{d}s \\
&\quad +\frac{w(t)}{2}  \int_\R \int_{\mathbb{S}^1} |\theta-e_1|^2   G_\kappa(|(t, \varphi(t) e_1)-(s,\varphi(s)\theta)|) \varphi(s)\mathrm{d}\theta  \mathrm{d}s .
\end{split}
\end{align}
\subsubsection{$D\cF_1(\varphi)$ is the Fr\'echet derivative of $\cF_1$}
We claim that the  linear operator $D\cF_1(\varphi): X\to Y$ is bounded.
Indeed, proceeding analogously as in Proposition \ref{pro:F-in-Holder}, we find that
\begin{align*}
\|F_\varphi\|_{C^{0,\alpha}(\R)}&\leq   C\|w'\|_{L^\infty(\R)} \|\varphi\|_{L^\infty(\R)} \leq C\|w\|_{C^{2,\alpha}(\R)} \|\varphi\|_{L^\infty(\R)}
\end{align*}
where we define
\begin{align*}
F_\varphi(w) &:=\int_{\mathbb{S}^1} \int_\mathbb{R}
(w(t)- w(s)) G_\kappa(|(t, \varphi(t) e_1)-(s,\varphi(s)\theta)|) \varphi(s)ds \mathrm{d} \theta.
\end{align*}
Meanwhile, since $G_k(r)= \frac{ e^{-\kappa r}}{r}$  we observe that
\begin{align*}
G_\kappa(|(t, \varphi(t) e_1)-(s,\varphi(s)\theta)|)
&=G_\kappa \big(\big((t-s)^2+(\varphi(t)-\varphi(s))^2+ \varphi(t)\varphi(s)|\theta-e_1|^2\big)^{1/2}\big)\\
&\leq
|\theta-e_1|^{-1}(\varphi(t)\varphi(s))^{-1/2}\, e^{-\kappa |t-s|}.
\end{align*}
Using this and the fact  $0<\delta = \inf_{t\in\mathbb{R}} \varphi(t)$
we get

\begin{align*}
\Big\|\frac{|w(t)|}{2} \int_{\mathbb{R}}\int_{\mathbb{S}^1} |\theta-e_1|^2& \varphi(s)G_\kappa(|(t, \varphi(t) e_1)-(s,\varphi(s)\theta)|) \mathrm{d} \theta \mathrm{d}s\Big\|_{C^{0,\alpha}(\R) }\\
 &  \leq 3\|w\|_{C^{0,\alpha}(\R)}\int_{\mathbb{R}}\int_{\mathbb{S}^1}\frac{|\theta-e_1|}{2} \Big(\frac{\varphi(s)}{\inf_{t\in\mathbb{R}}\varphi(t)} \Big)^{1/2}
 e^{-\kappa|r|} \mathrm{d} \theta \mathrm{d}r\\
&\leq
 6|\mathbb{S}^1| \kappa^{-1}\delta^{-1/2} \|w\|_{C^{2,\alpha}(\R)} \|\varphi\|^{1/2}_{L^\infty(\R)}.
\end{align*}
Combining both estimates proves the
boundedness of $D\cF_1(\varphi)$
\begin{align*}
\|D\cF_1(\varphi)(w) \|_{C^{0,\alpha}(\R)} \leq C\|w\|_{C^{2,\alpha}(\R)}.
\end{align*}
\noindent Before we complete the proof that $\cF_1$ is $C^\infty$ with Fr\'echet derivative $D\cF_1(\varphi)$ given by \eqref{eq:DefcF1}, we must establish the next Lemma \ref{lem:diff-cJ}. To this end, we define $\Lambda_0: C^{0,\beta}(\R)\times \R^3\to \R$  by
\begin{align*}
\Lambda_0(\varphi,t,r,p)=  (rp)^{-1}(\varphi(t)-\varphi(t-pr))=  \int_0^1\varphi'(t-\varrho pr )\mathrm{d}\varrho .
\end{align*}
So that for $\varphi \in C^{1, \beta}(\R)$ with $\beta\in (0,1]$ we have
\begin{align*}
|\Lambda_0(\varphi,t_1,r,p)-\Lambda_0(\varphi,t_2,r,p)|\leq 4 \|\varphi\|_{C^{1,\beta}(\R)} |t_1-t_2|^\beta.
\end{align*}
We now define
$$
p_\theta=|\theta-e_1|.
$$
Making the change of variable  $r=\frac{t-s}{p_\theta },$ we get the new expression
\begin{align}
D\cF_1(\varphi)(w)(t)
&=-\int_{\mathbb{S}^1}   \int_\R (w(t)- w(t-r  ))  \varphi(t-r) \overline{\cK}(\varphi,t,r,p_\theta) \mathrm{d}r  \mathrm{d}\theta   \nonumber\\
&\quad+\frac{w(t)}{2}\int_{\mathbb{S}^1}p_\theta^2 \int_\R    \varphi(t-rp_\theta) \cK(\varphi,t,r,p_\theta)  \mathrm{d}r \mathrm{d}\theta , \label{eq:cand-deriv-cF}
\end{align}
where $\cK,\overline{\cK}:\mathcal{O}\times \R^3\to \R$, our  kernels $\cK$ and $\overline{\cK}$  take the form
\begin{align*}
\overline{ \cK}(\varphi,t,r,p )
%
&= \frac{   \exp\left\{-{\kappa}\left\{  r^2+   r^2 \Lambda_0(\varphi,t,r,1)^2 + \varphi(t)\varphi(t- r) p^2    \right\}^{1/2} \right\}}{\left\{  r^2+
r^2 \Lambda_0(\varphi,t,r,1)^2 + \varphi(t)\varphi(t-r)p^2  \right\}^{1/2} },
\\
\cK(\varphi,t,r,p )
%
&= \frac{   \exp\left\{-{\kappa}{p}\left\{  r^2+  r^2 \Lambda_0(\varphi,t,r,p)^2+ \varphi(t)\varphi(t-p r)    \right\}^{1/2} \right\}}{\left\{  r^2+
r^2 \Lambda_0(\varphi,t,r,p)^2 + \varphi(t)\varphi(t-pr)  \right\}^{1/2} }.
\end{align*}

We then have to prove the following
\begin{lemma}
\label{lem:diff-cJ}
Let $\delta>0$, $\beta\in (0,1)$ and $\alpha\in (0,\beta)$. Define the set
\begin{align*}
\cB^\beta_\delta:=\{\varphi\in C^{1,\beta}(\R)\,:\, \varphi>\delta\}.
\end{align*}
Then
for every $w\in C^{0, \beta}(\R)$, the map
$\cJ_w:  \cB^\beta_\delta \to C^{0,\alpha}(\R)$
defined by
\begin{align*}
\cJ_w(\varphi)(t)
&=-\int_{\mathbb{S}^1}   \int_\R (w(t)- w(t-r  ))  \varphi(t-r) \overline{\cK}(\varphi,t,r, p_\theta) \mathrm{d}r  \mathrm{d}\theta   \nonumber\\
&\quad+\frac{w(t)}{2}\int_{\mathbb{S}^1}p_\theta^2 \int_\R    \varphi(t-rp_\theta) \cK(\varphi,t,r,p_\theta)  \mathrm{d}r \mathrm{d}\theta
\end{align*}
is of class $C^\infty$, where $p_\theta=|\theta-e_1|$.  Moreover, for every $k\in \N$, there exists a constant $c=c(\kappa,\delta,\alpha,\beta,k)>1$ such that
$$
\|D^k \cJ_w(\varphi)\|\leq c  \|w\|_{C^{0,\beta}(\R)}(1+\|\varphi\|_{C^{1,\beta}(\R)})^c.
$$

\end{lemma}
Once we have  the regularity  of $\cJ_w$ in this  Lemma, we easily deduce that    $\cF_1$ is smooth.
\begin{corollary}\label{cor:diff-f1}
The mapping $\cF_1: \mathcal{B}_\delta \to C^{0,\alpha}$ is  $C^\infty$, with $D\cF_1(\varphi)$ given as in $\eqref{eq:DefcF1}$.
\end{corollary}
\begin{proof}
We already know that the linear operator $w\mapsto D\cF_1(\varphi)(w):= \cJ_w(\varphi)$ is bounded.  Let  $w\in C^{2,\alpha}(\mathbb{R})\subset  C^{1,\beta}(\R)$, such that $\|w\|_{C^{2,\alpha}(\mathbb{R})}<\delta/2$.
By our formal computations in Section \ref{ss:can-deriv}, together with the fundamental theorem of calculus, we have
\begin{align*}
[\mathcal{ F}_1(\varphi+ w)-\mathcal{ F}_1(\varphi)- \mathcal{J}_w(\varphi )](t)
%
&= \int_0^1 [\mathcal{J}_w(\varphi+ \varrho  w) -  \mathcal{J}_w(\varphi )](t)\, d\varrho\\
%
%
&= \int_0^1 \varrho\int_0^1 D\mathcal{J}_w(\varphi+ \rho \varrho w)(w)  (t)\,d\rho d\varrho.
\end{align*}
  Since  $\|w\|_{L^\infty(\R)}\leq \|w\|_{C^{2,\alpha}(\mathbb{R})}<\delta/2$ we get, $w>-\delta/2$ so that  $ \varphi+\rho \varrho w>\delta/2 $ for $\rho,\varrho\in (0,1)$ that is  $\varphi+\rho \varrho w\in \mathcal{B}^\beta_{\delta/2}$ , it then follows from Lemma \ref{lem:diff-cJ} that
\begin{align*}
\|\mathcal{ F}_1(\varphi+ w)-\mathcal{ F}_1(\varphi)- \mathcal{J}_w(\varphi )\|_{C^{0,\alpha}(\mathbb{R})}&\leq  \int_0^1 \int_0^1 \| D\mathcal{J}_w(\varphi+ \rho \varrho  w)(w)\|_{C^{0,\alpha}(\mathbb{R})}\,  d \rho d\varrho\\
&\leq c  \, \|w\|_{C^{2,\alpha}(\mathbb{R})}^2(1+\|\varphi\|_{C^{2,\alpha}(\mathbb{R})}+\|w\|_{C^{2,\alpha}(\mathbb{R})} )^c.
\end{align*}
Therefore, we conclude that $\cF_1$ is differentiable on $\cB^\beta_{\delta/2}$ with $
D\cF_1(\varphi)(h)=\cJ_h(\varphi)$  since
\begin{align*}
\lim_{\|w\|_{C^{2,\alpha}(\R)}\to 0} \frac{\|\mathcal{ F}_1(\varphi+ w)-\mathcal{ F}_1(\varphi)- \mathcal{J}_w(\varphi )\|_{C^{0,\alpha}(\mathbb{R})}}{\|w\|_{C^{2,\alpha}(\mathbb{R})}} = 0.
\end{align*}
The $C^\infty$-character of $\cF_1$ on $\cB^\beta_{\delta}$ now follows from the one of $\cJ_w$ by Lemma \ref{lem:diff-cJ}, and the fact that $\delta$ is an arbitrary small positive number.
\end{proof}

\subsubsection{Proof of Lemma \ref{lem:diff-cJ} }
The proof of  Lemma \ref{lem:diff-cJ} will be inspired from \cite[Section 4 or even Section 5]{CFW18}. Comparison the situation in  \cite{CFW18},   it is worth noticing that  our kernel here is "less singular at the origin" and "decay faster at infinity".  However the main term of our  kernel given by $\frac{1}{r}e^{-\kappa r}$ is clearly not homogeneous and  not integrable at the origin.  We therefore have to study  estimates related to the kernel $\cK$ and $\overline{\cK}$ from which we will deduce the differentiability of $\cF_1$ between sharp H\"{o}lder spaces.

For a function $u: \R \to \R$, we use the notation
$$
[u; s_1,s_2]:= u(s_1)-u(s_2)\qquad \text{for $s_1,s_2 \in \R$,}
$$
and we note the obvious equality
\begin{align}
\label{eq:uv-s_1s_2}
[uv; s_1,s_2] = [u;s_1,s_2]v(s_1) + u(s_2)[v;s_1,s_2] \qquad \text{for $u,v: \R \to \R$, $s_1,s_2 \in \R$.}
\end{align}
As  in \cite[Section 4 ]{CFW18}, to prove the regularity of $\cJ_w$, it will be crucial to have estimates
related to the maps  $\cK $ and $\overline{\cK} $.
\begin{lemma}\label{lem:est-cK-2D}
Let $k\in \N$.
There exists a constant $ c=c(\kappa,\beta,\alpha,k,\delta)>1 $
such that  for all $(t,t_1,t_2,r)\in\R^4$, $|p|<2$ and  ${u}\in \cB^\beta_\delta$, we have
\begin{enumerate}
\item[(i)]
\begin{align}
\label{eq:Dk-K0-s}
\|  D_{u}^k \overline{\cK}({u},t,r,p    )
\|
&\leq
{c(1+ \|{u}\|_{C^{1,\beta}(\R)} )^{c}   }  \frac{  \exp(-\kappa|r| )}{   |r|}  ,
\\
\label{eq:Dk-K0-s_1s_2}
\| [D_{u}^k \overline{\cK}({u},\cdot ,r );t_1,t_2] \|&\leq
{c(1+ \|{u}\|_{C^{1,\beta}(\R)} )^{c}   \, |t_1-t_2|^\beta}   \frac{  \exp(-\kappa|r|) }{   |r|   }.
\end{align}
\item[(ii)]
\begin{align}
\label{eq:Dk-K1-s}
\|  D_{u}^k {\cK} ({u},t,r,p   )   &\|\leq
{c(1+ \|{u}\|_{C^{1,\beta}(\R)} )^{c}   }   \frac{  \exp(-\kappa|p|(1 + r^2)^{1/2} )  }{    (1 + r^2)^{1/2} },
\\
\label{eq:Dk-K1-s_1s_2}
\| [D_{u}^k {\cK} ({u},\cdot ,r);t_1,t_2] \|
&\leq
c(1+ \|{u}\|_{C^{1,\beta}(\R)} )^{c}   \, |t_1-t_2|^\beta \frac{ \exp(-\kappa|p|(1 + r^2)^{1/2} )}{    (1 + r^2)^{1/2}}.
\end{align}
\end{enumerate}
\end{lemma}
\begin{proof}
For $p\in (0,2)$ and $x>0$,  we define
\begin{align*}
g(x)=\frac{1}{x^{1/2}}\exp(-\kappa|p| x^{1/2}).
\end{align*}
Let $\ell\in\N$. Then there exists $c=c(\ell,\kappa) $ such that
for every  $p\in (0,2)$ and $r>0$,
\begin{align*}
|g^{(\ell)}(x)|\leq c  x^{-\frac{1+2\ell }{2}}\exp(-\kappa|p| x^{1/2}).
\end{align*}
The proof then is similar to the one of  \cite[Lemm 4.1]{CFW18}. We skip the details.
\end{proof}

\noindent The following  provide the desired estimates for higher derivatives of $\cJ_w$.
\begin{lemma}\label{lem:est-cand-deriv-2D}
Let  $\delta>0$, $\beta\in (0,1)$, $\alpha\in (0,\beta)$, ${u} \in \cB^\beta_\delta$ and $\psi,u, {u}_1,\dots, {u}_k \in C^{1,\beta}(\R)$ and $k\in \N \cup \{0\}$. We  define the functions $\cV_i: \R\to \R$ by
\begin{align*}
\cV_1(t)&= \int_{\mathbb{S}^1} \int_{\R}   (w(t)-w(t-r)) \psi(t-r)   D_{u}^k {\overline{\cK}} ({u},t ,r , p_\theta   ) [u_1,\dots,u_k] \mathrm{d}r  \mathrm{d}\theta
\\
\cV_2(t)&=  w(t) \int_{\mathbb{S}^1} p_\theta^2 \int_{\R}   \psi(t-r p_\theta)  D_{u}^k {\cK} ({u},t ,r   , p_\theta ) [u_1,\dots,u_k] \mathrm{d}r \mathrm{d}\theta .
\end{align*}
Then the following estimates hold
\begin{align}
\label{eq:est-F1-2D}
\|\cV_1\|_{C^{0, \alpha}(\R)}
&\leq  c(1+ \|{u}\|_{C^{1,\beta}(\R)} )^{c} \|w\|_{C^{0,\beta}(\R)}  \|\psi\|_{C^{1,\beta}(\R)}     \prod_{i=1}^k \|{u}_i\|_{C^{1, \beta}(\R)}
\\
\label{eq:est-F2-2D}
\|\cV_2\|_{C^{0,\beta}(\R)}
&\leq  c(1+ \|{u}\|_{C^{1,\beta}(\R)} )^{c}  \|w\|_{C^{0,\beta}(\R)}  \|\psi\|_{C^{1,\beta}(\R)}     \prod_{i=1}^k \|{u}_i\|_{C^{1, \beta}(\R)}.
\end{align}
\end{lemma}
\begin{proof}
By Lemma \ref{lem:est-cK-2D}(i), we have
\begin{align}
\label{eq:est-F1-2D-00}
\|\cV_1\|_{L^\infty(\R)} \leq  c(1+ \|{u}\|_{C^{1,\beta}(\R)} )^{c} \|w\|_{C^{0,\beta}(\R)}  \|\psi\|_{C^{1,\beta}(\R)}     \prod_{i=1}^k \|{u}_i\|_{C^{1, \beta}(\R)} .
\end{align}
Moreover, since
$$ | w(t_1)-w(t_2-r)-w(t_1) + w(t_2-r)|\leq2 \|w\|_{C^{0,\beta}(\R)}\min(|t_1-t_2|^\beta,|r|^\beta)  ,$$
 using inductively \eqref{eq:uv-s_1s_2},  we get
\begin{align*}
&|\cV_1(t_1)-\cV_1(t_2)|\leq  c(1+ \|{u}\|_{C^{1,\beta}(\R)} )^{c}  \|w\|_{C^{0,\beta}(\R)}  \|\psi\|_{C^{1,\beta}(\R)}     \prod_{i=1}^k \|{u}_i\|_{C^{1, \beta}(\R)}\\
&\times \left( \int_{|r|<|t_1-t_2|} \frac{\exp(-\kappa|r|)}{|r|^{1-\beta}}\mathrm{d}r + |t_1-t_2|^\beta  \int_{|r|>|t_1-t_2|} \frac{\exp(-\kappa|r|)}{|r|}\mathrm{d}r + |t_1-t_2|^\beta  \right) \\
&\leq c(1+ \|{u}\|_{C^{1,\beta}(\R)} )^{c}  \|w\|_{C^{0,\beta}(\R)}  \|\psi\|_{C^{1,\beta}(\R)}     \prod_{i=1}^k \|{u}_i\|_{C^{1, \beta}(\R)}\\
&\times (|t_1-t_2|^\beta+ |t_1-t_2|^\beta\log(|t_1-t_2|) ) .
\end{align*}
This with \eqref{eq:est-F1-2D-00}, give \eqref{eq:est-F1-2D}.
By Lemma \ref{lem:est-cK-2D}(ii) and a change of variable, we have
\begin{align}
\label{eq:est-F2-2D00}
\|\cV_2\|_{L^\infty(\R)}& \leq  c(1+ \|{u}\|_{C^{1,\beta}(\R)} )^{c}  \|w\|_{C^{0,\beta}(\R)}  \|\psi\|_{C^{1,\beta}(\R)}     \prod_{i=1}^k \|{u}_i\|_{C^{1, \beta}(\R)}\nonumber\\
&\times \int_{\mathbb{S}^1}p_\theta^{2}   \int_\R      G_{ \kappa}((s^2+\delta^2p_\theta^2)^{1/2} )\mathrm{d}s \mathrm{d}\theta  \nonumber\\
&\leq   c(1+ \|{u}\|_{C^{1,\beta}(\R)} )^{c}  \|w\|_{C^{0,\beta}(\R)}  \|\psi\|_{C^{1,\beta}(\R)}     \prod_{i=1}^k \|{u}_i\|_{C^{1, \beta}(\R)}.
\end{align}
Moreover, using inductively \eqref{eq:uv-s_1s_2}, we find that
\begin{align*}
|\cV_2(t_1)-\cV_2(t_2)|&
\leq  c(1+ \|{u}\|_{C^{1,\beta}(\R)} )^{c}  \|w\|_{C^{0,\beta}(\R)}  \|\psi\|_{C^{1,\beta}(\R)}     \prod_{i=1}^k \|{u}_i\|_{C^{1, \beta}(\R)}\\
&\times |t_1-t_2|^\beta \int_{\mathbb{S}^1}p_\theta^{2}   \int_\R      G_{ \kappa}((s^2+\delta^2p_\theta^2)^{1/2} )\mathrm{d}s \mathrm{d}\theta \\
& \leq  c(1+ \|{u}\|_{C^{1,\beta}(\R)} )^{c}  \|w\|_{C^{0,\beta}(\R)}  \|\psi\|_{C^{1,\beta}(\R)}     \prod_{i=1}^k \|{u}_i\|_{C^{1, \beta}(\R)} |t_1-t_2|^\beta.
\end{align*}
This with \eqref{eq:est-F2-2D00}, yield \eqref{eq:est-F2-2D}.
\end{proof}

\medskip
\textbf{Completion of the proof of   Lemma \ref{lem:diff-cJ} }
Since, by Lemma \ref{lem:est-cand-deriv-2D},   we have estimates of all possible candidates for the derivatives of $\cJ_w: \cB^\beta_\delta\to C^{0,\alpha}(\R)$, following the arguments in  \cite[Section 4 or   Section 5]{CFW18}, we immediately have that  $\cJ_w$ is of class $C^\infty$ in $\cB^\beta_\delta$. Moreover
\begin{align*}
\|D^k \cJ_w\| \leq   c(1+ \|{u}\|_{C^{1,\beta}(\R)} )^{c}  \|w\|_{C^{0,\beta}(\R)} .
\end{align*}
This completes the proof of the proposition.

\vspace{-1mm}


\end{document}